\documentclass[12pt]{amsart}

\usepackage{amssymb, amsmath}
\usepackage{amsthm}
\usepackage{mathrsfs}
\usepackage[pdfborder={0 0 0 [0 0 ]},bookmarksdepth=3, bookmarksopen=true,unicode=true]{hyperref}
\usepackage{cite}
\usepackage{enumitem}
\usepackage{cases}
\usepackage[font={it}, margin=1cm]{caption}
\usepackage{verbatim}
\usepackage[capitalize]{cleveref}
\usepackage{color}
\usepackage[ps,all,arc,rotate]{xy}

\usepackage{tikz}
\usetikzlibrary{decorations.markings, decorations.pathreplacing, positioning,arrows, matrix, decorations.pathmorphing, shapes.geometric, calc}
\usetikzlibrary{backgrounds, decorations, cd}

\newtheorem{theorem}{Theorem}[section]

\newtheorem{lemma}[theorem]{Lemma} 
\newtheorem{proposition}[theorem]{Proposition} 
\newtheorem{corollary}[theorem]{Corollary} 
\newtheorem{convention}[theorem]{Convention} 
\newtheorem{thmletter}{Theorem}

\newtheorem{corolletter}[thmletter]{Corollary} 

\newtheorem{example}[theorem]{Example}

\newcommand{\p}[1]{\noindent {\newline\bf #1.}}
\newcommand{\out}{\operatorname{Out}}
\newcommand{\aut}{\operatorname{Aut}}

\newcommand{\im}{\operatorname{im}}
\newcommand{\syl}{\operatorname{Syl}}

\newcommand{\bGamma}{\mathbf{\Gamma}}

\newcommand{\CSA}{CSA$_{\mathbb{Q}}$}
\newcommand{\BS}{$\operatorname{BS}$}
\newcommand{\Z}{$\mathbb{Z}_{\max}$}
\newcommand{\ZC}{$\mathcal{Z}_{\max}$}

\DeclareMathOperator{\FT}{\operatorname{FT}}

\newcounter{dawidcomments}
\newcommand{\dawid}[1]{\textbf{\color{red}(D\arabic{dawidcomments})} \marginpar{\scriptsize\raggedright\textbf{\color{red}(D\arabic{dawidcomments})Dawid: }#1}
\addtocounter{dawidcomments}{1}}
\newcounter{alancomments}
\newcommand{\alan}[1]{\textbf{\color{blue}(A\arabic{alancomments})} \marginpar{\scriptsize\raggedright\textbf{\color{blue}(A\arabic{alancomments})Alan: }#1}
\addtocounter{alancomments}{1}}
\newcounter{gilescomments}

\title[JSJ decompositions and polytopes for one-relator groups]
{JSJ decompositions and polytopes for two-generator one-relator groups}
\author{Giles Gardam}
\address{
Mathematisches Institut, Universit\"at Bonn, Endenicher Allee 60, 53115 Bonn,
Germany}
\email{gardam@math.uni-bonn.de}
\author{Dawid Kielak}
\address{
University of Oxford, Oxford, OX2 6GG,
UK}
\email{kielak@maths.ox.ac.uk}
\author{Alan D. Logan}
\address{
Heriot-Watt University, Edinburgh, EH14 4AS,
UK}
\email{A.Logan@hw.ac.uk}


\begin{document}

\begin{abstract}
We provide a direct connection between the $\mathcal{Z}_{\max}$ (or essential) JSJ decomposition and the Friedl--Tillmann polytope of a hyperbolic two-generator one-relator group with abelianisation of rank $2$.

We deduce various structural and algorithmic properties, like the existence of a quadratic-time algorithm computing the \ZC{}-JSJ decomposition of such groups.
\end{abstract}
\maketitle

\[ \mbox{\emph{Dedicated to the memory of Stephen J. Pride}}\]
\vspace{0.5cm}

\section{Introduction}
\label{introduction}
A two-generator one-relator group is a group that admits a presentation of the form $\langle a, b\mid R\rangle$.
One-relator groups are a cornerstone of Geometric Group Theory (see, for example, the classic texts in combinatorial group theory \cite{mks} \cite{L-S} or the early work of Magnus \cite{magnus1930freiheitssatz} \cite{magnus1932wordproblem}), and they continue to have fruitful interactions with, for example, $3$-dimensional topology and knot theory. Motivated by the Thurston norm of a $3$-manifold, Friedl--Tillmann introduced a polytope for presentations $\langle a, b\mid R\rangle$ with $R\in F(a, b)'$ \cite{Friedl2015Two}, which has subsequently been shown to be a group invariant \cite{HennekeKielak2020} (another proof of this fact follows from the work of Friedl--L\"uck~\cite{FriedlLueck2017} combined with a more recent result of Jaikin-Zapirain--L\'opez-\'Alvarez on the Atiyah conjecture \cite{Jaikin-ZapirainLopez-Alvarez2018}).

In this paper we focus on hyperbolic one-relator groups, and more generally one-relator groups with no Baumslag--Solitar subgroups
$\langle a, t\mid t^{-1}a^mt=a^n\rangle$; we call these algebraic generalisations of hyperbolic one-relator groups \emph{\BS{}-free one-relator groups}.
JSJ-theory rose to prominence due to Sela's work on $\mathcal{Z}_{\max}$-JSJ decompositions (originally called ``essential'' JSJ decompositions) of hyperbolic groups, which are graph of groups decompositions encoding all the ``important'' virtually-$\mathbb{Z}$ splittings of the group;
these decompositions are significant because they are a group invariant (up to certain moves), and they, and hence the splittings they encode, govern for example the model theory \cite{Sela2009} and (coarsely) the outer automorphism group
\cite{Sela1997Structure}
\cite{levitt2005automorphisms}
of the group.
Moreover, computing \ZC{}-JSJ decompositions is a key step in the algorithm to solve the isomorphism problem for hyperbolic groups
\cite{sela1995isomorphism}
\cite{Dahmani2008ToralIso}
\cite{dahmani2011isomorphism}.
The notation \ZC{} relates to a certain class of maximal virtually cyclic subgroups, but for one-relator groups these subgroups are necessarily infinite cyclic; therefore one can view our results as being about ``\Z{}-JSJ decompositions'', but we maintain the \ZC{} notation for consistency with the literature.

Our main theorem connects \ZC{}-JSJ decompositions and Friedl--Tillmann polytopes~\cite{Friedl2015Two}.
This connection is significant because it means that, under our assumptions, the \ZC{}-JSJ decomposition of $\langle a, b\mid R\rangle$ can be understood simply by investigating the relator $R$, which yields fast algorithmic results (see below).
Furthermore, this  gives a connection between JSJ decompositions and these polytopes, as JSJ decompositions are refinements of \ZC{}-JSJ decompositions \cite[Section 9.5]{Guirardel2017JSJ}.

Friedl--Tillmann polytopes are only defined for two-generator one-relator groups, and such groups are of central importance in the theory of one-relator groups.
For example,
it is a classical result that every one-relator group embeds into a two-generator one-relator group \cite[Corollary 4.10.1]{mks}, while
Louder--Wilton have pointed out that all known examples of pathological one-relator groups have two generators \cite{louder2018negative}.
Moreover, \BS{}-free two-generator one-relator groups are particularly important:
It is a famous conjecture of Gersten that every \BS{}-free one-relator group is hyperbolic (see \cite[p.~228]{Gersten1992}, \cite[Remark, p.~734]{Allcock1999homological}), and Linton  proved that this conjecture reduces to the two generator case \cite[Corollary 5.7]{Linton2024}, i.e.\ every \BS{}-free one-relator group is hyperbolic if and only if every \BS{}-free two-generator one-relator group is hyperbolic.

We say that a one-ended group has \emph{trivial \ZC{}-JSJ decomposition} if it has a \ZC{}-JSJ decomposition which is a single vertex with no edges, and \emph{non-trivial} otherwise (see Convention \ref{conv:JSJ}).

\begin{thmletter}[Theorem \ref{thm:flexibilityHyperbolicBODYVERSION}]
\label{thm:flexibilityHyperbolic}
Let $G$ be a \BS{}-free group admitting a two-generator one-relator presentation $\mathcal{P}=\langle a, b\mid R\rangle$ with $R\in F(a, b)'\setminus\{1\}$.
The following are equivalent.
\begin{enumerate}
\item\label{flexibilityHyperbolic:1} $G$ has non-trivial \ZC{}-JSJ decomposition.
\item\label{flexibilityHyperbolic:2} There exists a word $T$ of shortest length in the $\aut(F(a, b))$-orbit of $R$ such that $T\in\langle a, b^{-1}ab\rangle$ but $T$ is not conjugate to $[a, b]^{k}$ for any $k\in\mathbb{Z}$.
\item\label{flexibilityHyperbolic:3} The Friedl--Tillmann polytope of $\mathcal{P}$ is a straight line, but not a single point.
\end{enumerate}
\end{thmletter}

It is unclear how intrinsic hyperbolicity is to this theorem.
If Gersten's conjecture is true, then the result is simply for hyperbolic two-generator one-relator groups.
On the other hand, an analogue of Theorem \ref{thm:flexibilityHyperbolic} will hold for any class of one-relator groups for which Theorem \ref{thm:JSJFormTF}, below, on Friedl--Tillmann polytopes is applicable, and which satisfy a certain description of splittings first given by Kapovich--Weidmann \cite[Theorem 3.9]{Kapovich1999structure}, and extended by the authors \cite[Corollary 8.6]{gardam2021algebraically}.
(Indeed, \ZC{}-JSJ decompositions are not required. They just give context and a convenient language for our results, which instead can be stated in terms of ``essential $\mathbb{Z}$-splittings'', as in Theorem \ref{thm:KWquote}.)

We now illustrate Theorem \ref{thm:flexibilityHyperbolic} with an example.

\begin{example}
Let $G$ be the group defined by the presentation $\langle a, b\mid(a^2b^2a^{-1}b^{-1}a^{-1}b^{-1})^n\rangle$ where $n>1$.
By using Whitehead's algorithm, it can be seen that (\ref{flexibilityHyperbolic:2}) of Theorem \ref{thm:flexibilityHyperbolic} does not hold, and so $G$ has trivial \ZC{}-JSJ decomposition.
Alternatively one can consider the {Friedl--Tillmann polytope}, which we see from Figure \ref{fig:PolytopeExample} is a triangle.
Therefore, (\ref{flexibilityHyperbolic:3}) of Theorem \ref{thm:flexibilityHyperbolic} does not hold, and so $G$ has trivial \ZC{}-JSJ decomposition.
\begin{figure}[h]
\centering
\begin{minipage}{0.24\textwidth}
  \begin{tikzpicture}
\draw[step=1cm,gray,very thin, dashed] (-0.4,-0.4) grid (2.4,2.4);
\draw (-0.4,0) -- (2.4, 0);
\draw (0,-0.4) -- (0,2.4);
\draw[color=blue] (0,0) -- (2.1,0.1) -- (2.1,2.1) -- (0.9,2.1) -- (0.9, 1.1) -- (0.1, 1.1)
-- (0.1, 0.1) -- (1.9,0.1) -- (1.9,1.9) -- (1.1,1.9) -- (1.1, 0.9) -- (0.15, 0.9) -- (0, 0);
  \end{tikzpicture}
\end{minipage}
\begin{minipage}{0.24\textwidth}
  \begin{tikzpicture}
\draw[step=1cm,gray,very thin, dashed] (-0.4,-0.4) grid (2.4,2.4);
\draw (-0.4,0) -- (2.4, 0);
\draw (0,-0.4) -- (0,2.4);
\filldraw[fill=green!20!white, fill opacity=0.5, draw=green!50!black] (0,0) -- (2,0) -- (2,2) -- (1,2)  -- (0, 1) -- (0, 0);
\filldraw[draw=green!50!black, fill=white] (0,0) circle (0.1cm);
\filldraw[draw=green!50!black, fill=white] (2,0) circle (0.1cm);
\filldraw[draw=green!50!black, fill=white] (2,2) circle (0.1cm);
\filldraw[draw=green!50!black, fill=white] (1,2) circle (0.1cm);
\filldraw[draw=green!50!black, fill=white] (0,1) circle (0.1cm);
  \end{tikzpicture}
  \end{minipage}
  \begin{minipage}{0.24\textwidth}
  \begin{tikzpicture}
\draw[step=1cm,gray,very thin, dashed] (-0.4,-0.4) grid (2.4,2.4);
\draw (-0.4,0) -- (2.4, 0);
\draw (0,-0.4) -- (0,2.4);
\draw[green!50!black] (-0.05,-0.05) -- (2,-0.05) -- (2,2) -- (1,2)  -- (-0.05, 1) -- (-0.05, -0.05);
\draw[red] (0,0) -- (0.9,0) -- (0.9,0.9) -- (0,0.9)  -- (0, 0);
\draw[red] (1,1) -- (1.9,1) -- (1.9,1.9) -- (1,1.9)  -- (1, 1);
\draw[red] (1,0) -- (1.9,0) -- (1.9,0.9) -- (1,0.9)  -- (1, 0);
\filldraw[draw=red, fill=white] (0,0) circle (0.1cm);
\filldraw[draw=red, fill=white] (1,0) circle (0.1cm);
\filldraw[draw=red, fill=white] (1,1) circle (0.1cm);
  \end{tikzpicture}
    \end{minipage}
    \begin{minipage}{0.24\textwidth}
  \begin{tikzpicture}
\draw[step=1cm,gray,very thin, dashed] (-0.4,-0.4) grid (2.4,2.4);
\draw (-0.4,0) -- (2.4, 0);
\draw (0,-0.4) -- (0,2.4);
\filldraw[fill=red!20!white, fill opacity=0.5, draw=red] (0,0) -- (1,0) -- (1,1) -- (0, 0);
\filldraw[draw=red, fill=white] (0,0) circle (0.1cm);
\filldraw[draw=red, fill=white] (1,0) circle (0.1cm);
\filldraw[draw=red, fill=white] (1,1) circle (0.1cm);
  \end{tikzpicture}
\end{minipage}

  \caption{To obtain the Friedl--Tillmann polytope,
trace the reduced word $(a^2b^2a^{-1}b^{-1}a^{-1}b^{-1})^n$ on the $ab$-plane to obtain a closed loop $\gamma$, as in the first diagram (this is independent of $n$).
Take the convex hull of $\gamma$, as in the second diagram; this is a polytope $P'$.
Then take the bottom-left corner of all squares contained in $\gamma$ that touch the vertices of $P'$, as in the third diagram.
The Friedl--Tillmann polytope $P$ is the polytope with these points as vertices, as in the fourth diagram.\\\\
Note that the Friedl--Tillmann polytope is in fact a ``marked'' polytope, but we only care about the shape so we have omitted these details from this example.
In the third diagram we took the bottom-left corner of the squares; this is different from Friedl and Tillmann who take the centre points of these squares, but this is not an issue because the polytope is only well-defined up to translation.
}
  \label{fig:PolytopeExample}
\end{figure}
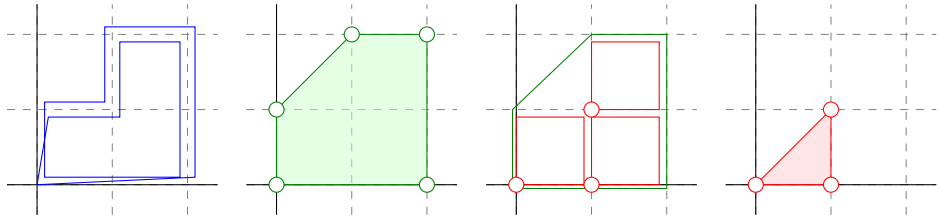
\end{example}

The proof of Theorem \ref{thm:flexibilityHyperbolic} splits into two cases: either $G$ is a one-relator group with torsion, or is torsion-free.
The difficulty lies in the torsion-free case; see in particular Section \ref{sec:torsionFree}.

If $G$ does not have $\mathbb{Z}^2$ abelianisation, so $R\not\in F(a, b)'$, then the Friedl--Tillmann polytope is less useful (it is \emph{always} a straight line!).
However, we can still ask if (\ref{flexibilityHyperbolic:1}) and (\ref{flexibilityHyperbolic:2}) from Theorem \ref{thm:flexibilityHyperbolic} are equivalent.
To prove this, we could try to apply the machinery underlying these polytopes, which is that of the universal $L^2$ torsion, but unfortunately this lies beyond our current understanding of these $L^2$ invariants.
We can however use more classical techniques to prove this equivalence for one-relator groups with torsion.

\p{One-relator groups with torsion}
The easier case in the proof of Theorem \ref{thm:flexibilityHyperbolic} is that of one relator groups with torsion. Such groups are characterised by having presentations $\langle \mathbf{x}\mid S^n\rangle$ where $n>1$ \cite[Proposition II.5.18]{L-S}. They are always hyperbolic \cite[Theorem IV.5.5]{L-S} (Nyberg-Brodda has put this fact in its historical context \cite{brodda2020b}), and as such these groups are of great interest as test-cases for both hyperbolic groups and one-relator groups. For example, one-relator groups with torsion are residually finite \cite{wise2012riches}, while it is an open problem of Gromov whether all hyperbolic groups are residually finite; similarly, one-relator groups with torsion were shown to be coherent in \cite{louder2020one}, while it took another five years to show that in fact all one-relator groups are coherent \cite{JaikinLinton2023}.

Our results for one-relator groups with torsion include the case of $R\not\in F(a, b)'$.
A \emph{primitive element} of $F(a, b)$ is an element which is part of a basis for $F(a, b)$.
\begin{thmletter}
[Theorem \ref{thm:flexibilityWithTorsionBODYVERSION}]
\label{thm:flexibilityWithTorsion}
Let $G$ be a group admitting a two-generator one-relator presentation $\mathcal{P}=\langle a, b\mid R\rangle$ where $R= S^n$ in $F(a, b)$ with $n>1$ maximal and $S\in F(a, b)$ is non-trivial and non-primitive.
The following are equivalent.
\begin{enumerate}
\item\label{flexibilityHyperbolic:1} $G$ has non-trivial \ZC{}-JSJ decomposition.
\item\label{flexibilityHyperbolic:2} There exists a word $T$ of shortest length in the $\aut(F(a, b))$-orbit of $S$ such that $T\in\langle a, b^{-1}ab\rangle$ but $T$ is not conjugate to $[a, b]^{\pm 1}$.
\end{enumerate}
\end{thmletter}

The conditions on $S$ are because \ZC{}-JSJ decompositions are only meaningful for one-ended groups, and if $S$ is primitive or trivial then the group defined by $\langle a, b\mid S^n\rangle$ is not one-ended (it is a free product of cyclic groups, $\mathbb{Z}\ast C_n$ or $\mathbb{Z\ast Z}$). In contrast, all the groups in Theorem \ref{thm:flexibilityHyperbolic} are one-ended.

\p{Forms of {\boldmath \ZC{}}-JSJ decompositions}
The following corollary describes the \ZC{}-JSJ decompositions of the groups from Theorems \ref{thm:flexibilityHyperbolic} and \ref{thm:flexibilityWithTorsion} as HNN-extensions of one-relator groups.

In the torsion-free case, a previous result of the authors describes such \ZC{}-JSJ decompositions as HNN-extensions \cite[Corollary 8.6]{gardam2021algebraically}; Corollary \ref{corol:JSJform} further says that the base groups are one-relator groups, which is surprising.
In the corollary,
we view $R$ as a power $S^n$ for $n\geqslant1$ maximal. We view $R$ like this because Theorem \ref{thm:flexibilityWithTorsion} deals with the root $S$ of the relator $R=S^n$. We lose nothing by doing this since if the group in the corollary is torsion-free then $n=1$ and $R=S$.
We use $|W|$ to denote the length of a word $W\in F(\mathbf{x})$.
\begin{corolletter}[Corollary \ref{corol:JSJformBODYVERSION}]
\label{corol:JSJform}
Let the group $G$ and the presentation $\mathcal{P}$ be as in Theorem \ref{thm:flexibilityHyperbolic} or \ref{thm:flexibilityWithTorsion}, and write $R=S^n$ for $n\geqslant1$ maximal.
Suppose that $G$ has non-trivial \ZC{}-JSJ decomposition $\mathbf{\Gamma}$. Then the graph underlying $\mathbf{\Gamma}$ consists of a single rigid vertex and a single loop edge.
Moreover, the corresponding HNN-extension has vertex group $\langle a, y\mid T_0^n(a, y)\rangle$, stable letter $b$, and attaching map given by $y=b^{-1}ab$, where $T_0(a, b^{-1}ab)$ is the word $T$ from the theorem. Finally, $|T_0|<|S|$.
\end{corolletter}

In the case of Theorem \ref{thm:flexibilityWithTorsion}, the base group $\langle a, y\mid T_0(a, y)\rangle$ also satisfies the conditions of Theorem \ref{thm:flexibilityWithTorsion}. Therefore, the ``$|T_0|<|S|$'' condition gives a strong accessibility result, similar to a general result of Louder--Touikan \cite{Louder2017Strong}, but where all the groups involved are one-relator groups. If we are in the case of Theorem \ref{thm:flexibilityHyperbolic} then $T_0(a, y)$ may not be in the derived subgroup, and so we do not obtain the analogous result.

\begin{example}
\label{ex:JSJexample}
Set $G=\langle a, b\mid (a^{-2}b^{-1}a^2b)^n\rangle$, $n>1$.
Then $G$ is an HNN-extension of the group $H=\langle a, y\mid (a^{-2}y^2)^n\rangle$, with attaching map $y=b^{-1}ab$.
This obvious decomposition of $G$ as an HNN-extension corresponds to its \ZC{}-JSJ decomposition, by Corollary \ref{corol:JSJform}.
\end{example}

This example illustrates an important point: we can spot \ZC{}-JSJ decompositions of the groups in Theorem \ref{thm:flexibilityHyperbolic} or \ref{thm:flexibilityWithTorsion}, since if a decompositions looks like a \ZC{}-JSJ decomposition then it is indeed a \ZC{}-JSJ decomposition (see Theorems \ref{thm:KWquote} and \ref{thm:JSJclassificationONEREL}).
Now, the group $G$ from Example \ref{ex:JSJexample} can also be viewed as an HNN-extension of the group $H'=\langle a, z\mid (a^{-2}z)^n\rangle$, with attaching map $z=b^{-1}a^2b$. This is not a \ZC{}-splitting (as $\langle z\rangle$ is not maximal in $G$), but it could be a JSJ decomposition of $G$. However, we were unable to prove analogues of Theorems \ref{thm:KWquote} and \ref{thm:JSJclassificationONEREL} for JSJ decompositions, and so we are unable to conclude that this splitting is in fact a JSJ decomposition.

\p{Algorithmic consequences}
Computing \ZC{}-JSJ decompositions has implications for the isomorphism problem, gives information about outer automorphism groups, and is potentially an important first step for many algorithmic questions about the elementary theory of hyperbolic groups \cite{Dahmani2008ToralIso}.
However, current algorithms for computing JSJ decompositions or \ZC{}-JSJ decompositions have bad {computational} complexity.

For example, algorithms of Barrett \cite{Barrett2018Computing}, Dahmani--Groves \cite{Dahmani2008ToralIso}, Dahmani--Guirardel \cite{dahmani2011isomorphism}, and Dahmani--Touikan \cite{Dahmani2019Deciding} work in very general settings, but each of them when applied to hyperbolic groups has no recursive bound on its time complexity. For the algorithms of Barrett, Dahmani--Groves, and Dahmani--Guirardel, this is because they require computation of the hyperbolicity constant $\delta$, and there is no recursive bound on the time complexity for computing $\delta$ (as hyperbolicity is undecidable). The algorithm of Dahmani--Touikan requires a solution to the word problem, but all known general solutions for hyperbolic groups require preprocessing for which there is no recursive bound on the time complexity (for example, computing an automatic structure or Dehn presentation). Even if we assume an oracle gives us $\delta$, these algorithms are brute force algorithms, with each proceeding by {detecting} a splitting and then searching blindly through all presentations of the given group to find some presentation which realises the detected splitting; this procedure clearly has an awful time complexity.
As far as the authors are aware, the only algorithm which computes the JSJ decompositions for a class of hyperbolic groups and which has a known (reasonable) bound on its time complexity is due to Suraj Krishna \cite{SurajKrishna2020Immersed}, but here the computed bound is doubly exponential and is not in general applicable to one-relator groups.

Theorems \ref{thm:flexibilityHyperbolic} and \ref{thm:flexibilityWithTorsion} can be applied to give fast algorithms for both detecting and computing \ZC{}-JSJ decompositions in our setting. Firstly, there is a quadratic-time algorithm to find the \ZC{}-JSJ decomposition of a given group (essentially, the algorithm is to compute all of the shortest possible elements of the $\aut(F(a, b))$-orbit of the relator $R$).

\begin{corolletter}[Corollary \ref{corol:JSJcompBODYVERSION}]
\label{corol:JSJcomp}
There exists an algorithm with input a presentation $\mathcal{P}=\langle a, b\mid R\rangle$ of a group $G$ from Theorem \ref{thm:flexibilityHyperbolic} or \ref{thm:flexibilityWithTorsion}, and with output the \ZC{}-JSJ decomposition of $G$.

This algorithm terminates in $O(|R|^2)$-steps.
\end{corolletter}

In the case of Theorem \ref{thm:flexibilityHyperbolic}, the polytope allows us to detect a non-trivial \ZC{}-JSJ decomposition in linear time (essentially, the algorithm is to draw the Friedl--Tillmann polytope).
\begin{corolletter}[Corollary \ref{corol:JSJdetectBODYVERSION}]
\label{corol:JSJdetect}
There exists an algorithm with input a presentation $\mathcal{P}=\langle a, b\mid R\rangle$ of a group $G$ from Theorem \ref{thm:flexibilityHyperbolic}, and with output {\ttfamily\upshape yes} if the group $G$ has non-trivial \ZC{}-JSJ decomposition and {\ttfamily\upshape no} otherwise.

This algorithm terminates in $O(|R|)$-steps.
\end{corolletter}

Detecting non-trivial \ZC{}-JSJ decompositions is useful in its own right. We demonstrate this with the following application of Corollaries \ref{corol:JSJcomp} and \ref{corol:JSJdetect} to hyperbolic groups.

\begin{corolletter}[Corollary \ref{corol:OutdetectBODYVERSION}]
\label{corol:Outdetect}
There exists an algorithm with input a presentation $\mathcal{P}=\langle a, b\mid R\rangle$ of a hyperbolic group $G$ from Theorem \ref{thm:flexibilityHyperbolic} or \ref{thm:flexibilityWithTorsion} that determines which one of the following three possibilities holds: the outer automorphism group of $G$ is finite, is virtually $\mathbb{Z}$, or is isomorphic to $\operatorname{GL}_2(\mathbb{Z})$.

If $G$ is as in Theorem \ref{thm:flexibilityHyperbolic}, this algorithm terminates in $O(|R|)$-steps. Else, it terminates in $O(|R|^2)$-steps.
\end{corolletter}

\p{Relationships between outer automorphism groups}
Write $G_k$ for the group defined by $\langle a, b\mid S^k\rangle$, where $k\geqslant1$ is maximal and $S$ is fixed.
In general there is very little relationship between $\out(G_1)$ and $\out(G_n)$ for $n>1$. For example, if $S=b^{-1}a^2ba^{-4}$ then $G_1=\operatorname{BS}(2, 4)$ has non-finitely generated outer automorphism group \cite{collins1983automorphisms}, while $\out(G_n)$ for $n \geqslant 2$ is virtually-$\mathbb{Z}$ \cite{Logan2016Outer}.
In contrast, Theorems \ref{thm:flexibilityHyperbolic} and \ref{thm:flexibilityWithTorsion} imply that if $S\in F(a, b)'$ and $G_1$ is hyperbolic then the triviality of the \ZC{}-JSJ decomposition of $G_k$ depends solely on the word $S$; the exponent $k$ is irrelevant. We can then apply the relationship between \ZC{}-JSJ decompositions and outer automorphism groups to see that $\out(G_m)$ and $\out(G_n)$ are commensurable for $m, n\geqslant1$. Our next corollary says that this relationship is much stronger.

\begin{corolletter}[Corollary \ref{corol:OutCommensurabilityBODYVERSION}]
\label{corol:OutCommensurability}
Write $G_k$ for the group defined by $\langle a, b\mid S^k\rangle$, where $k\geqslant1$ is maximal. If $S\in F(a, b)' \setminus \{1\}$ and $G_1$ is hyperbolic then:
\begin{enumerate}
\item\label{OutCommensurability:1} $\out(G_m)\cong\out(G_n)$ for all $m, n>1$.
\item\label{OutCommensurability:2} $\out(G_n)$ embeds with finite index in $\out(G_1)$.
\end{enumerate}
\end{corolletter}
The isomorphism of (\ref{OutCommensurability:1}) is essentially already known, and holds under the more general restriction that $S$ is non-primitive \cite{Logan2016Outer}; we include it for completeness. All the maps here are extremely natural, and in particular the embedding $\out(G_n)\hookrightarrow\out(G_1)$ is induced by the natural map $G_n\rightarrow G_1$ with kernel normally generated by $S$.

\noindent{\newline\bfseries What about JSJ decompositions?}
Most of our main results can be stated in terms of JSJ decompositions, but we state them in terms of \ZC{}-JSJ decompositions as this gives a smoother exposition. Nothing is lost by doing this as \ZC{}-JSJ decompositions encode all the information we are interested in and motivated by (outer automorphism groups, isomorphism problem), and in fact we gain Corollary \ref{corol:JSJcomp}, which we are unable to rephrase in terms of JSJ decompositions.
Additionally, when we rephrase Corollary \ref{corol:JSJform} in these terms then the phrase ``$T_0(a, b^{-1}ab)$ is the word $T$ from the theorem'' is replaced with ``$T_0(a, b^{-1}ab)$ is some word'', so we lose the explicit connection with the theorems.

\p{Outline of the paper}
In Section \ref{sec:background} we build a theory of \ZC{}-JSJ decompositions applicable to the non-hyperbolic groups in Theorem \ref{thm:flexibilityHyperbolic}, and we prove a useful theorem which allows us to spot \ZC{}-JSJ decompositions of these groups (Theorem \ref{thm:KWquote}).
In Section \ref{sec:withTorsion} we prove Theorem \ref{thm:flexibilityWithTorsion}.
In Section \ref{sec:torsionFree} we prove our main technical result involving polytopes, Theorem \ref{thm:JSJFormTF}, which takes as input a pair of compatible presentations of some group $G$, one having the form $\langle a, b\mid R\rangle$ with $R\in F(a, b)'$ and the other the form $\langle a, b\mid\mathbf{s}\rangle$ with $\mathbf{s}\subset\langle a, b^{-1}ab\rangle$, and proves that the Friedl--Tillmann polytope of $G$ is a straight line.
In Section \ref{sec:RGTheorem} we prove Theorem \ref{thm:flexibilityHyperbolic}.
In Section \ref{sec:mainProofs} we prove Corollaries \ref{corol:JSJform}--\ref{corol:OutCommensurability}.

\p{Acknowledgements}
We are very grateful to Nicholas Touikan for several helpful comments and suggestions.
We thank the referee for their helpful comments.
This work has received funding from the European Research Council (ERC) under the European Union's Horizon 2020 research and innovation programme, Grant agreement No. 850930, 
from the European Union (ERC, SATURN, 101076148),
from the UK's Engineering and Physical Sciences Research Council (EPSRC), grant EP/R035814/1,
and from the Deutsche Forschungsgemeinschaft (DFG, German Research Foundation) -- Project-ID 427320536 -- SFB 1442, as well as under Germany’s Excellence Strategy EXC 2044--390685587 and EXC-2047/1 -- 390685813.


\section{JSJ-theory and algebraically hyperbolic groups}
\label{sec:background}
JSJ-theory plays a key role in the theory of hyperbolic groups \cite{rips1997cyclic,sela1995isomorphism,levitt2005automorphisms,Sela2009,dahmani2011isomorphism}.
Mirroring the JSJ decomposition of $3$-manifolds, the JSJ and \ZC{}-JSJ decompositions of a one-ended hyperbolic group are graph of groups decompositions where all edge groups are virtually-$\mathbb{Z}$, and are (in an appropriate sense) unique if certain vertices, called ``flexible vertices'', are treated as pieces to be left intact and not decomposed.
\ZC{}-JSJ decompositions were originally called ``essential JSJ decompositions'' by Sela, and they contain all the important information used in the applications of JSJ-theory cited above.
We refer the reader to Guirardel--Levitt's monograph for further background and motivation, in particular their Section 9.5 on \ZC{}-JSJ decompositions \cite{Guirardel2017JSJ}.

In this section we define JSJ decompositions and \ZC{}-JSJ decompositions.
We then build a theory of \ZC{}-JSJ decompositions for certain algebraically hyperbolic (\CSA{}) groups, as defined below, and which is applicable to the BS-free one-relator groups considered in this paper.
Our main result here is Theorem \ref{thm:KWquote}, which can be summarised by: {when a decomposition looks like a \ZC{}-JSJ decomposition, it is indeed a \ZC{}-JSJ decomposition}.

This section deals with graph of groups splittings, so we now fix some conventions and notation (based on Serre's book \cite{trees}). In a graph $\Gamma$, every edge $e$, with initial and terminal vertices $\iota(e)$ and $\tau(e)$, respectively, has an associated reverse edge $\overline{e}$ such that $\iota(\overline{e})=\tau(e)$ and $\tau(\overline{e})=\iota(e)$. We use $\bGamma$ to denote a graph of groups, with a connected underlying graph $\Gamma$, associated \emph{vertex groups} (or \emph{vertex stabilisers}) $\{G_v\mid v\in V({\Gamma})\}$ and \emph{edge groups} (or \emph{edge stabilisers}) $\{G_e\mid e\in E({\Gamma}), G_e=G_{\overline{e}}\}$, and set of monomorphisms $\theta_e: G_e\rightarrow G_{\iota(e)}$.


\p{JSJ decompositions}
We fix a finitely generated group $G$, and a family $\mathcal{A}$ of subgroups of $G$.
(It is usual to assume that $\mathcal{A}$ is closed under conjugating and taking subgroups, but the main class we focus on, \ZC{} subgroups, is not closed under these operations.)

An \emph{$\mathcal{A}$-tree of $G$} is a tree $T$ equipped with an action of $G$, written $G\curvearrowright T$, whose edge stabilisers are in $\mathcal{A}$.
A subgroup $H\leqslant G$ is \emph{elliptic} in $T$ if it fixes a point in $T$ (and hence is contained in a vertex stabiliser of $T$), and \emph{universally $\mathcal{A}$-elliptic} if it is elliptic in every $\mathcal{A}$-tree of $G$.
An $\mathcal{A}$-tree $T$ is \emph{universally $\mathcal{A}$-elliptic} if its edge stabilisers are universally $\mathcal{A}$-elliptic.
An $\mathcal{A}$-tree $T$ \emph{dominates} the $\mathcal{A}$-tree $T'$ if every subgroup of $G$ which is elliptic in $T$ is elliptic in $T'$.
An \emph{$\mathcal{A}$-JSJ tree $T$ of $G$} is an $\mathcal{A}$-tree such that:
\begin{enumerate}
\item $T$ is universally $\mathcal{A}$-elliptic, and
\item $T$ dominates every other universally $\mathcal{A}$-elliptic tree $T'$.
\end{enumerate}
The quotient graph of groups $\bGamma$, with underlying graph $\Gamma =T/G$, is an \emph{$\mathcal{A}$-JSJ decomposition of $G$}.
A vertex of an $\mathcal{A}$-tree $T$ is \emph{elementary} if its group is in $\mathcal{A}$, while a non-elementary vertex is \emph{rigid} if its group is universally $\mathcal{A}$-elliptic, and \emph{flexible} otherwise.
Flexible vertices are the ``pieces to be left intact and not decomposed'' mentioned above, and these play a very minor role in this paper.

We shall principally consider four such families $\mathcal{A}$:
\begin{enumerate}
\item The family of $\mathcal{VC}$ subgroups consists of virtually-$\mathbb{Z}$ subgroups of $G$.
All the trees we consider are $\mathcal{VC}$, so we shall abbreviate ``$\mathcal{VC}$-tree'', ``$\mathcal{VC}$-JSJ tree'' and ``$\mathcal{VC}$-JSJ decomposition'' to simply \emph{tree}, \emph{JSJ tree} and \emph{JSJ decomposition}, respectively.
\item The family of $\mathcal{Z}$ subgroups consists of virtually-$\mathbb{Z}$ subgroups of $G$ with infinite centre.
For all the groups we are dealing with, any $\mathcal{Z}$ subgroup is infinite cyclic \cite{Karrass1971subgroups, Cebotar1971subgroups}, so we often use $\mathbb{Z}$ in place of $\mathcal{Z}$.
\item The family of \ZC{} subgroups consists of maximal virtually-$\mathbb{Z}$ subgroups of $G$ with infinite centre.
For the same reason as above, we often use \Z{} in place of \ZC{}.
\item The family of $\widetilde{\mathbb{Z}}$ subgroups consists of those $\mathbb{Z}$ subgroups $H$ such that if $K$ is abelian with $H\leqslant K\leqslant G$ then $K$ is also infinite cyclic.
\end{enumerate}

The theory of \ZC{}-JSJ decompositions for hyperbolic groups is especially powerful because there are canonical structures: for each group $G$ there is a unique ``deformation space'' of \ZC{}-JSJ trees, and there is a canonical \ZC{}-JSJ tree \cite[Section 9.5]{Guirardel2017JSJ}.
We wish to extend these results to \ZC{}-JSJ trees of the groups in Theorem \ref{thm:flexibilityHyperbolic}.
To facilitate this, we consider algebraically hyperbolic groups.

\subsection{Algebraically hyperbolic groups}
\label{sec:AHgroups}
A torsion-free group is \emph{algebraically hyperbolic}, written \CSA{}, if it is commutative-transitive and contains no subgroups of the form $H\rtimes\mathbb{Z}$ with $H$ non-trivial locally cyclic.
We are interested in algebraically hyperbolic groups because torsion-free one-relator groups are algebraically hyperbolic if and only if they are \BS{}-free \cite[Theorem B]{gardam2021algebraically}, and because we can readily study their \ZC{}-JSJ decompositions, or equivalently (by torsion-free-ness) their \Z{}-JSJ decompositions.

The results in this section and at the start of the next are already known for hyperbolic groups, so the reader who is only interested in hyperbolic groups may skip to Proposition \ref{prop:RGZmaxJSJ} of Section \ref{sec:TwoGenJSJ}.

A group $G$ is called \emph{CSA} (for \emph{conjugately separated abelian}) if every maximal abelian subgroup of $G$ is malnormal.
Algebraically hyperbolic groups are CSA \cite[Theorem D]{gardam2021algebraically}, and so our starting point for the study of their \Z{}-JSJ decompositions is the existing theory of JSJ decompositions of CSA groups \cite[Theorem 9.5]{Guirardel2017JSJ}.
(The notation \CSA{} comes from the fact that a torsion-free group is algebraically hyperbolic if and only if it is a CSA group where every non-trivial abelian subgroup is locally infinite cyclic, i.e.\ CSA and abelian subgroups embed into $(\mathbb{Q}, +)$). 

We have kept this section short, and only prove the few results we need.
In particular, we focus on \CSA{} groups rather than CSA groups, and we only prove that \Z{}-JSJ decompositions exist under certain conditions.
We suspect that this section can be extended into a complete \Z{}-JSJ theory for CSA groups.

\subsubsection{Existence}
We start by proving that certain \CSA{} groups have \Z{}-JSJ decompositions.
To do this we follow the analogous existence proof for hyperbolic groups (see for example \cite[Section 9.5]{Guirardel2017JSJ} or \cite[Sections 4.3 \& 4.4]{dahmani2011isomorphism}), although an extra step is needed to deal with the existence of cyclic subgroups not contained in $\widetilde{\mathbb{Z}}$.

We restrict ourselves to those \CSA{} groups whose JSJ decompositions have no flexible vertices.
This sidesteps subtleties in the proof for hyperbolic groups (see the discussion following Remark 9.29 of \cite{Guirardel2017JSJ}, or \cite[Remark 4.13]{dahmani2011isomorphism}),
but is the only situation used in this paper.

\p{The cyclic collapsed tree of cylinders {\boldmath$T_{\operatorname{JSJ}}$}}
Let $G$ be a finitely generated one-ended \CSA{} group.
Then $G$ admits a canonical JSJ tree $T_{\operatorname{JSJ}}$, namely the \emph{cyclic collapsed tree of cylinders} of Guirardel--Levitt \cite[Theorem 9.5]{Guirardel2017JSJ}.
This is a bipartite tree with vertex set $V_0(T_{\operatorname{JSJ}})\sqcup V_1(T_{\operatorname{JSJ}})$, where $V_0(T_{\operatorname{JSJ}})$ consists of non-elementary vertices and $V_1(T_{\operatorname{JSJ}})$ consists of elementary vertices.
Note that edge groups are subgroups of finite index in the vertex groups of elementary vertices, essentially since non-trivial subgroups of $\mathbb Z$ are of finite index.
The action $G\curvearrowright T_{\operatorname{JSJ}}$ is \emph{minimal}, that is $T_{\operatorname{JSJ}}$ has no proper $G$-invariant sub-tree (this is an ambient assumption in Guirardel--Levitt's monograph \cite[Section 1.2]{Guirardel2017JSJ}).
We adapt this tree to obtain a canonical \Z{}-JSJ tree $T_{\mathbb{Z}_{\max}}$, which also has certain nice properties (e.g.\ invariance under group automorphisms).

There are two obstacles to overcome for $T_{\operatorname{JSJ}}$.
Firstly, elementary vertex groups may be contained in non-cyclic abelian subgroups, and secondly, elementary vertex groups may not be maximal abelian.

\p{The {\boldmath$\widetilde{\mathbb{Z}}^c$}-collapse of {\boldmath$T_{\operatorname{JSJ}}$}}
Let $C$ be a $\mathbb{Z}$ subgroup of a \CSA{} group $G$.
Zorn's lemma guarantees the existence of a maximal abelian subgroup $A$ of $G$ containing $C$.
Let $A'$ be another maximal abelian subgroup of $G$ containing $C$.
Take $a \in A'$.
As $a$ centralises $C$, we have that $C$ is a non-trivial subgroup of $aAa^{-1}$, and so by malnormality of $A$ we have that $a\in A$.
Hence $A' = A$, and we have shown that if $C$ is a $\mathbb{Z}$ subgroup of a \CSA{} group $G$ then there exists a unique maximal abelian subgroup containing $C$; we will denote this subgroup by $\hat{C}$.
Note that $\hat{C}$ is locally cyclic, although may not be cyclic, and that the family $\widetilde{\mathbb{Z}}$ consists of those $\mathbb{Z}$ subgroups $C$ such that $\hat{C}$ is still a $\mathbb{Z}$ subgroup.

We will use $\widetilde{\mathbb{Z}}^c$ to denote the complement of the family $\widetilde{\mathbb{Z}}$ within $\mathbb{Z}$-subgroups, so the family consisting of $\mathbb{Z}$ subgroups $C$ of $G$ such that $\hat{C}$ is not finitely generated.

The \emph{$\widetilde{\mathbb{Z}}^c$-collapse} of $T_{\operatorname{JSJ}}$, written $T_{\widetilde{\mathbb{Z}}}$, is the quotient of $T_{\operatorname{JSJ}}$ by the smallest equivalence relation such that for vertices $v, w$ of $T_{\operatorname{JSJ}}$, $v\sim w$ if $G_v\cap G_w\in\widetilde{\mathbb{Z}}^c$.
Equivalently, the tree $T_{\widetilde{\mathbb{Z}}}$ is the $G$-tree constructed from $T_{\operatorname{JSJ}}$ as follows:
For all elementary vertices $v$ of $T_{\operatorname{JSJ}}$ such that $\hat{G}_v$ is non-cyclic, collapse the subgraph of $T_{\operatorname{JSJ}}$ consisting of $v$ and all adjacent edges and their endpoints
(i.e.\ the \emph{star} of $v$)
to a single point.
As orbits of vertices have conjugate vertex groups we have $\hat{G}_v\cong\hat{G}_{g\cdot v}$ for all $g\in G$, so if $v$ is collapsed in this way then so is every vertex in its orbit, and so $T_{\widetilde{\mathbb{Z}}}$ is a $G$-tree, and hence a $\widetilde{\mathbb{Z}}$-tree.
We also see that $T_{\widetilde{\mathbb{Z}}}$ inherits the bipartite structure from $T_{\operatorname{JSJ}}$.
Since the action $G\curvearrowright T_{\operatorname{JSJ}}$ is minimal, the action $G\curvearrowright T_{\widetilde{\mathbb{Z}}}$ is also minimal, and the first description of $T_{\widetilde{\mathbb{Z}}}$ gives us that, because $T_{\operatorname{JSJ}}$ is canonical, the tree $T_{\widetilde{\mathbb{Z}}}$ is canonical.

\begin{lemma}
\label{lem:LCcollapse}
Let $G$ be a finitely generated one-ended \CSA{} group whose JSJ-tree has no flexible vertices.
Then the $\widetilde{\mathbb{Z}}^c$-collapse $T_{\widetilde{\mathbb{Z}}}$ of the cyclic collapsed tree of cylinders $T_{\operatorname{JSJ}}$ of $G$ is a canonical $\widetilde{\mathbb{Z}}$-JSJ tree of $G$, and has no flexible vertices.
\end{lemma}

\begin{proof}
We first prove that $T_{\widetilde{\mathbb{Z}}}$ is universally $\widetilde{\mathbb{Z}}$-elliptic.
The collapsing used to get from $T_{\operatorname{JSJ}}$ to $T_{\widetilde{\mathbb{Z}}}$ reduces the set of edge stabilisers, but does not add any new edge stabilisers, and so edge stabilisers of $T_{\widetilde{\mathbb{Z}}}$ are edge stabilisers of $T_{\operatorname{JSJ}}$.
As $T_{\operatorname{JSJ}}$ is universally $\mathbb{Z}$-elliptic, its edge stabilisers are universally $\mathbb{Z}$-elliptic, and hence the edge stabilisers of $T_{\widetilde{\mathbb{Z}}}$ are also universally $\mathbb{Z}$-elliptic.
As every $\widetilde{\mathbb{Z}}$-tree is a $\mathbb{Z}$-tree, and as $T_{\widetilde{\mathbb{Z}}}$ is a $\widetilde{\mathbb{Z}}$-tree, it follows that the edge stabilisers of $T_{\widetilde{\mathbb{Z}}}$ are universally $\widetilde{\mathbb{Z}}$-elliptic, as claimed.

We now prove that vertex groups of non-elementary vertices of $T_{\widetilde{\mathbb{Z}}}$ are universally $\widetilde{\mathbb{Z}}$-elliptic, i.e.\ non-elementary vertices are rigid.
So consider a $\widetilde{\mathbb{Z}}$-tree $S$, and let $v$ be a non-elementary vertex of $T_{\widetilde{\mathbb{Z}}}$.
As $T_{\widetilde{\mathbb{Z}}}$ is constructed from $T_{\operatorname{JSJ}}$ by collapsing stars of $\widetilde{\mathbb{Z}}^c$ elementary vertices to single points, the subgroup $G_v$ acts on $T_{\operatorname{JSJ}}$ so that there is a minimal $G_v$-invariant subtree $T_{\operatorname{JSJ}}^{(v)}$ with vertex
(resp.\ edge)
stabilisers which are themselves vertex
(resp.\ edge)
stabilisers in the action on $G$ on $T_{\operatorname{JSJ}}$ (as opposed to subgroups of vertex, resp.\ edge, stabilisers), and with an inherited bipartite structure of elementary and non-elementary vertices.
Therefore, all elementary vertex stabilisers in the action $G_v\curvearrowright T_{\operatorname{JSJ}}^{(v)}$ are in $\widetilde{\mathbb{Z}}^c$, and all non-elementary vertex stabilisers are universally elliptic (as they are rigid in the action of $G$ on the JSJ tree $T_{\operatorname{JSJ}}$).
Suppose $G_{v_i}$ is a non-elementary vertex group of the action $G_v\curvearrowright T_{\operatorname{JSJ}}^{(v)}$.
Then $G_{v_i}$ is universally elliptic, and hence universally $\widetilde{\mathbb{Z}}$-elliptic, and so elliptic in the $\widetilde{\mathbb{Z}}$-tree $S$.
Suppose $G_{v_i}$ is an elementary vertex group of the action $G_v\curvearrowright T_{\operatorname{JSJ}}^{(v)}$.
Now let $e_j$ be an edge with $\iota(e_j)=v_i$.
As $T_{\operatorname{JSJ}}$ is universally elliptic, and as $G_{e_j}$ is an edge stabiliser of $T_{\operatorname{JSJ}}$, we have that $G_{e_j}$ is elliptic in $S$, i.e.\ fixes a point of $S$.
As $G_{v_i}$ is cyclic with generator $x$ such that $x^n$ fixes a point of $S$ for some $n\geqslant1$ (specifically, $G_{e_j}=\langle x^n\rangle$), we have that $G_{v_i}$ is elliptic in $S$ \cite[Proposition 25]{trees}.
As $S$ is a $\widetilde{\mathbb{Z}}$-tree but $G_{v_i}$ is not a $\widetilde{\mathbb{Z}}$ subgroup, $G_{v_i}$ in fact fixes a unique vertex of $S$.
Therefore, every vertex group $G_{v_i}$ in the action $G_v\curvearrowright T_{\operatorname{JSJ}}^{(v)}$ fixes a vertex of $S$.
On the other hand, none of these groups $G_{v_i}$ are in $\widetilde{\mathbb{Z}}$, but $S$ is a $\widetilde{\mathbb{Z}}$-tree, so no vertex group $G_{v_i}$ fixes an edge of $S$.
Now suppose $G_{v_i}$ and $G_{v_j}$ are adjacent vertices in the graph of groups decomposition of $G_v$ corresponding to its action on $T_{\operatorname{JSJ}}^{(v)}$.
Then $G_{v_i}$ and $G_{v_j}$ fix vertices $w_i, w_j$ of $S$ respectively, and the intersection $G_{v_i}\cap G_{v_j}$ fixes the geodesic $[w_i, w_j]$.
However, as one of $G_{v_i}$ or $G_{v_j}$ is in $\widetilde{\mathbb{Z}}^c$, we have that $G_{v_i}\cap G_{v_j}$ is in $\widetilde{\mathbb{Z}}^c$, and hence is not an edge stabiliser of the $\widetilde{\mathbb{Z}}$-tree $S$.
It follows that $w_i=w_j$, and inductively we see that each $G_{v_i}$ fixes the same point of $S$, and hence $G_v$ is elliptic in $S$ as required.

Finally, we prove that $T_{\widetilde{\mathbb{Z}}}$ dominates every other universally $\widetilde{\mathbb{Z}}$-elliptic tree $S$.
The collapsing used to get from $T_{\operatorname{JSJ}}$ to $T_{\widetilde{\mathbb{Z}}}$ reduces the set of elementary vertex groups, but does not add any new elementary vertex groups, and so elementary vertex groups of $T_{\widetilde{\mathbb{Z}}}$ are elementary vertex groups of $T_{\operatorname{JSJ}}$.
As $S$ is a $\widetilde{\mathbb{Z}}$-tree, it is a $\mathbb{Z}$-tree, and so is dominated by $T_{\operatorname{JSJ}}$.
Therefore, the elementary vertex groups of $T_{\operatorname{JSJ}}$ are elliptic in $S$, and hence the elementary vertex groups of $T_{\widetilde{\mathbb{Z}}}$ are also elliptic in $S$.
For non-elementary vertices, note that their stabilisers are rigid, and so are elliptic in every $\widetilde{\mathbb{Z}}$-tree, and hence are elliptic in $S$.
It follows that $T_{\widetilde{\mathbb{Z}}}$ dominates $S$, as required.
\end{proof}

\p{The {\boldmath\Z{}}-fold of {\boldmath$T_{\widetilde{\mathbb{Z}}}$}}
We now define the \emph{\Z{}-fold} of $T_{\widetilde{\mathbb{Z}}}$, written $T_{\mathbb{Z}_{\max}}$, which is the tree we require for the rest of this paper; under our assumptions, it is a canonical $\mathbb{Z}_{\max}$-JSJ tree and has no flexible vertices.

Given two distinct edges $e_1, e_2$ of a tree $S$ such that $\iota(e_1)=\iota(e_2)$, we can \emph{fold} these edges together by identifying $\tau(e_1)$ with $\tau(e_2)$, $e_1$ with $e_2$, and the reverse edge $\overline{e}_1$ with $\overline{e}_2$, and we obtain a new tree $S/[e_1=e_2]$.
If $G$ acts on $S$ and $\mathcal{E}\subset E(S)$ is some set of edges of $S$ with the same initial vertex then an \emph{orbit-fold} of the edges $\mathcal{E}$ is the tree $S/[G\curvearrowright \mathcal{E}]$ obtained by folding together all the edges in $g\cdot \mathcal{E}$ for every $g\in G$. This ensures that $G$ acts on the new tree $S/[G\curvearrowright \mathcal{E}]$.

Let $T_{\widetilde{\mathbb{Z}}}'$ be the quotient of $T_{\widetilde{\mathbb{Z}}}$ by the smallest equivalence relation such that, for all edges $e$ of $T_{\widetilde{\mathbb{Z}}}$ and all $h\in\hat{G}_e$, we have $h\cdot e\sim e$.
We shall use a constructive definition of $T_{\widetilde{\mathbb{Z}}}'$, which is equivalent \cite[Proof of Lemma 9.27]{Guirardel2017JSJ} (see also \cite[Section 4.3]{dahmani2011isomorphism}).
The tree $T_{\widetilde{\mathbb{Z}}}'$ is the $G$-tree constructed from $T_{\widetilde{\mathbb{Z}}}$ by iteratively orbit-folding together edges as follows:
Let $e$ be an edge such that $G_e\neq\hat{G}_e$. If one of the endpoints of $e$ is fixed by $\hat{G}_e$,
then orbit-fold together the set of edges in the $\hat{G}_e$-orbit of $e$.
If not, then as $\hat{G}_e$ is cyclic with generator $x$, say, such that $x^n$ fixes an edge for some $n\geqslant1$ (specifically, $G_e=\langle x^n\rangle$), we have that the subtree $\operatorname{Fix}(\hat{G}_e)$ fixed by $\hat{G}_e$ is non-empty \cite[Proposition 25]{trees}.
Let $e'$ be the first edge in the shortest path joining $\operatorname{Fix}(\hat{G}_e)$ to $e$, so $\iota(e')\in\operatorname{Fix}(\hat{G}_e)$ but $\tau(e')\not\in\operatorname{Fix}(\hat{G}_e)$.
Then orbit-fold together the set of edges in the $\hat{G}_{e'}$-orbit of $e'$.
The first description of $T_{\widetilde{\mathbb{Z}}}'$ gives us that, because $T_{\widetilde{\mathbb{Z}}}$ is canonical, the tree $T_{\widetilde{\mathbb{Z}}}'$ is canonical.
The second description of $T_{\widetilde{\mathbb{Z}}}'$ gives us that the edge stabilisers of $T_{\widetilde{\mathbb{Z}}}'$ are contained in the set $\{\hat{G}_e\mid e\in E(T_{\widetilde{\mathbb{Z}}})\}$, i.e.\ are \Z{}-subgroups which contain an edge stabiliser of $T_{\widetilde{\mathbb{Z}}}$, and so $T_{\widetilde{\mathbb{Z}}}'$ is a \Z{}-tree.
The \emph{\Z{}-fold} of $T_{\widetilde{\mathbb{Z}}}$, written $T_{\mathbb{Z}_{\max}}$, is the minimal $G$-invariant subtree of $T_{\widetilde{\mathbb{Z}}}'$; this tree is again both canonical and a \Z{}-tree.

\begin{lemma}
\label{lem:CSAJSJtree}
Let $G$ be a finitely generated one-ended \CSA{} group whose JSJ trees have no flexible vertices.
Then
the {\Z{}}-fold $T_{\mathbb{Z}_{\max}}$ of {$T_{\widetilde{\mathbb{Z}}}$}
is a canonical \Z{}-JSJ tree of $G$, and has no flexible vertices.
\end{lemma}

\begin{proof}
We first prove that $T_{\mathbb{Z}_{\max}}$ is universally \Z{}-elliptic.
Recall that edge stabilisers of $T_{\widetilde{\mathbb{Z}}}'$, and hence of $T_{\mathbb{Z}_{\max}}$, are contained in the set $\{\hat{G}_e\mid e\in E(T_{\widetilde{\mathbb{Z}}})\}$.
So, let $G_f$ be an edge stabiliser of $T_{\mathbb{Z}_{\max}}$, so $G_f=\hat{G}_e$ for some $e\in E(T_{\widetilde{\mathbb{Z}}})$, and let $S$ be an arbitrary \Z{}-tree of $G$.
Note that $S$ is also a $\widetilde{\mathbb{Z}}$-tree of $G$, as \Z{} subgroups are also $\widetilde{\mathbb{Z}}$.
As $T_{\widetilde{\mathbb{Z}}}$ is universally $\widetilde{\mathbb{Z}}$-elliptic, $G_e$ is universally $\widetilde{\mathbb{Z}}$-elliptic, and hence $G_e$ fixes some point $p$ of $S$.
As $S$ is a \Z{}-tree, and as $\hat{G}_e$ is the unique \Z{}-subgroup of $G$ containing $G_e$, it follows that $\hat{G}_e$ also fixes the point $p$ of $S$.
As the edge $f\in E(T_{\mathbb{Z}_{\max}})$ and the \Z{}-tree $S$ were arbitrary, we have that $T_{\mathbb{Z}_{\max}}$ is universally \Z{}-elliptic, as claimed.

We now prove that vertex groups of non-elementary vertices of $T_{\mathbb{Z}_{\max}}$ are elliptic in every \Z{}-tree $S$, i.e.\ non-elementary vertices are rigid.
As above, note that $S$ is also a $\widetilde{\mathbb{Z}}$-tree of $G$.
Let $v$ be a non-elementary vertex of $T_{\mathbb{Z}_{\max}}$.
Then there exists a non-elementary vertex $w$ in $T_{\widetilde{\mathbb{Z}}}$ and edges $f_1, \ldots, f_m$ adjacent to $w$ such that $G_{v}$ has stabiliser a multiple amalgam ${G}_{v}=G_{w}(\ast_{G_{f_i}}\hat{G}_{f_i})_{i=1}^m$ \cite[Lemma 4.10]{dahmani2011isomorphism}.
By Lemma \ref{lem:LCcollapse}, ${w}$ is a rigid vertex of $T_{\widetilde{\mathbb{Z}}}$, and so $G_w$ is elliptic in the $\widetilde{\mathbb{Z}}$-tree $S$.
Now, as $S$ is a \Z{}-tree, for each edge $f_i$ adjacent to $w$ the subgroup $\hat{G}_{f_i}$ must fix the same points as ${G}_{f_i}$, and so as ${G}_{w}$ is elliptic we also have that $G_{v}$ is elliptic.
As $S$ is an arbitrary \Z{}-tree, $v$ is rigid as required.

Finally, we prove that $T_{\mathbb{Z}_{\max}}$ dominates every other universally \Z{}-elliptic tree $S$.
Let $v$ be a vertex of $T_{\mathbb{Z}_{\max}}$.
Suppose $v$ is elementary.
Then any edge $f\in E(T_{\mathbb{Z}_{\max}})$ adjacent to $v$ satisfies $G_f\leqslant G_v$ and $\hat{G}_f=\hat{G}_v$.
Therefore, as $G_f=\hat{G}_f$ because $T_{\mathbb{Z}_{\max}}$ is a \Z{}-tree, we have that $G_v=G_f$.
As $T_{\mathbb{Z}_{\max}}$ is universally \Z{}-elliptic, $G_v=G_f$ is elliptic in every \Z{}-tree, and hence in every universally \Z{}-elliptic tree, as required for domination.
Suppose $v$ is non-elementary.
Then $v$ is rigid, and so is elliptic in every \Z{}-tree, and hence in every universally \Z{}-elliptic tree, as required for domination.
\end{proof}

\subsubsection{Uniqueness of deformation space}
JSJ trees or \Z{}-JSJ trees are not in general unique, but instead unique up to certain moves. More formally:
The \emph{deformation space} of a tree $T$ is the set of trees $T'$ such that $T$ dominates $T'$ and $T'$ dominates $T$. Equivalently, two trees are in the same deformation space if and only if they have the same elliptic subgroups.
Elements of a deformation space are connected by moves of a certain type \cite[Remark 2.13]{Guirardel2017JSJ} \cite[Theorem 1.1]{Clay2009Whitehead}.
Therefore, the following says that \Z{}-JSJ decompositions of a \CSA{} group are unique, up to the moves cited above.
The proof is clear because, by the definition of \Z{}-JSJ trees, they pairwise dominate one another and so all lie in the same deformation space.
\begin{proposition}
\label{prop:CSAJSJdefspace}
Let $G$ be a finitely generated one-ended \CSA{} group. If $G$ has a \Z{}-JSJ tree, then all \Z{}-JSJ trees of $G$ lie in the same deformation space $\mathcal{D}_{\mathbb{Z}_{\max}}$.
\end{proposition}

\subsection{Two-generated groups}
\label{sec:TwoGenJSJ}
We now apply the above results and discussions to the \Z{}-JSJ decompositions of one-ended two-generator \CSA{} groups.
In particular, we prove both existence and uniqueness of \Z{}-JSJ decomposition for these groups, and classify their structure (see Proposition \ref{prop:JSJsEverywhere}).
This lets us define, in Convention \ref{conv:JSJ}, \emph{the} \Z{}-JSJ decomposition of a one-ended two-generator \CSA{} group, and of a two-generator one-relator group with torsion.
Theorem \ref{thm:KWquote} then allows us to readily determine if a given splitting is, in fact, the \Z{}-JSJ decomposition of such a \CSA{} group.


\p{Existence}
We prove existence of \Z{}-JSJ decompositions by applying Lemma \ref{lem:CSAJSJtree} to the following lemma.
The proof of the lemma uses the canonical JSJ tree $T$ of $G$ \cite[Theorem 9.5]{Guirardel2017JSJ} which we previously used in the proof of Lemma \ref{lem:CSAJSJtree}. However, here we have no restriction on flexible vertices so our description changes slightly:
The canonical JSJ tree $T$ is a bipartite tree with vertex set $V_0(T)\sqcup V_1(T)$, where $V_0(T)$ consists of rigid and flexible vertices, and $V_1(T)$ consists of elementary vertices.

\begin{lemma}
\label{lem:RGnoflexible}
Let $G$ be a one-ended two-generator \CSA{} group. Then no JSJ tree of $G$ has a flexible vertex.
\end{lemma}

\begin{proof}
As $G$ is \CSA{}, it has a (canonical) JSJ tree $T$ \cite[Theorem 9.5]{Guirardel2017JSJ}.
This tree can be altered to give a new tree $\overline{T}_c$ which has no elementary vertices: on the level of $T/G$,
for each elementary vertex $v$ pick an incident edge $e$ and collapse the subgraph $(v, e, \tau(e))$ to a point.
The quotient graph $\overline{T}_c/G$ either consists of a single vertex and no edges, or a single vertex and a single loop edge \cite[Corollary 8.6]{gardam2021algebraically}.

Then by reversing the above moves, and using the bipartite structure of $T$, we have that $T/G$ consists of a single non-elementary vertex $v$, possibly an elementary vertex $u$ with incident edges $e, f$ each connected to $v$ (this corresponds to the loop edge), and possibly some elementary vertices of degree one each connected to $v$.

It follows that the vertex group $G_v$ is one-ended and two-generated \cite[Propositions 8.1 and 8.3]{gardam2021algebraically} (we apply one of these citations for each elementary vertex).
Suppose $v$ is flexible.
Then $G_v$ is the fundamental group of a compact surface \cite[Theorem 9.5]{Guirardel2017JSJ}.
It follows that $G_v$ must contain $\mathbb{Z}^2$ (for example because such groups are Fuchsian and by applying the classification of two-generated Fuchsian groups \cite[Lemma 1]{Rosenberger1986generating}), which contradicts $G$ being \CSA{}.
Therefore, $T$ does not have a flexible vertex.

The result now follows as, by Proposition \ref{prop:CSAJSJdefspace}, any JSJ-tree has the same elliptic subgroups as $T$, so has a flexible vertex if and only if $T$ has a flexible vertex.
\end{proof}

We now summarise what we know so far, whilst replacing the \Z{} notation with the \ZC{} notation.
For torsion-free hyperbolic groups, the first paragraph of this proposition is standard, while the second is due to Kapovich and Weidmann \cite[Proof of Theorem 3.11]{Kapovich1999structure}.

\begin{proposition}
\label{prop:RGZmaxJSJ}
Let $G$ be a one-ended two-generator \CSA{} group. Then $G$ has a canonical \ZC{}-JSJ tree $T_{\mathcal{Z}_{\max}}$, and all \ZC{}-JSJ trees of $G$ lie in the same deformation space $\mathcal{D}_{\mathcal{Z}_{\max}}$.

Moreover, no \ZC{}-JSJ tree of $G$ has a flexible vertex.
\end{proposition}

\begin{proof}
By Lemma \ref{lem:RGnoflexible}, we can apply Lemma \ref{lem:CSAJSJtree} to get a canonical \ZC{}-JSJ tree $T_{\mathcal{Z}_{\max}}$ for $G$, and this tree has no flexible vertices. The result then follows by Proposition \ref{prop:CSAJSJdefspace}, and because one \ZC{}-JSJ tree has a flexible vertex if and only if they all have a flexible vertex.
\end{proof}

\p{Uniqueness and structure}
We therefore have a unique deformation space of \ZC{}-JSJ trees, but we wish to strengthen this to obtain uniqueness of trees.
To do this, we require an additional property.

An \emph{essential \ZC{}-tree} is a \ZC{}-tree such that no edge group has finite index in an adjacent vertex group, while an \emph{essential \ZC{}-splitting} of a group is a corresponding graph of groups decomposition.
We define \emph{essential $\mathbb{Z}$-trees} and \emph{essential $\mathbb{Z}$-splittings} analogously, for $\mathbb{Z}$-trees.
We now prove Proposition \ref{prop:JSJsEverywhere}, which says that essential \ZC{}-trees are unique for our groups.
Note that a \ZC{}-JSJ decomposition with no elementary vertices is an essential \ZC{}-splitting.

The following says that (up to conjugation), there is a single essential \ZC{}-tree of $G$, and that this is a \ZC{}-JSJ tree of $G$.
This uniqueness is unusual, and is obtained by a result of Levitt on \emph{reduced} trees, which are trees $S$ where if an edge $e$ has the same stabiliser as one of its endpoints, then both endpoints of $e$ are in the same $G$-orbit of $S$, i.e.\ $e$ projects onto a loop in the quotient graph $S/G$ \cite[Theorem 1]{Levitt2005Characterizing} (see also \cite[Theorem 5.13]{dahmani2011isomorphism} for hyperbolic groups).
If $G$ contains no $\mathbb{Z}^2$ subgroups (e.g.\ is a hyperbolic or \CSA{} group), a \ZC{}-tree $S$ of $G$ is reduced if and only if it is essential;
an essential tree has no elementary vertices so is necessarily reduced,
while if $S$ is a reduced non-essential \ZC{}-tree then $T/G$ contains a loop edge with endpoints an elementary vertex, which by maximality of the edge groups gives a $\mathbb{Z}^2$ subgroup of $G$.

We now give the specific form of Levitt's result we will use, which assumes malnormal edge stabilisers; this assumption holds for \CSA{} groups, and for one-relator groups with torsion.
A $G$-tree $T$ is \emph{rigid} if it is the unique reduced tree in its deformation space.

\begin{lemma}
\label{lem:fixingPaths}
    Suppose $T$ is a reduced $G$-tree such that the quotient graph $T/G$ consists of a single vertex with a single loop edge.
    Suppose moreover $G$ is not an ascending HNN-extension, and that edge stablisers are malnormal in $G$.
    Then $T$ is rigid.
\end{lemma}

\begin{proof}
    Suppose that $g\in G$ fixes a two-edge path $p=(v_0, e, v_1, f, v_2)$ of $T$.
    Then $g\in G_e\cap G_f$.
    Assume that $g$ is non-trivial.
     As $T/G$ has only one edge, there exists an element $x\in G$ which takes $f$ to either $e$ or the reverse edge $\overline{e}$.
    Observing that $G_{e}=G_{\overline{e}}$ yields
    \[
    G_e
    =
    G_{xf}
    =
    x^{-1}G_fx,
    \]
    and so $g \in G_f \cap x^{-1}G_fx$. 
By malnormality of edge stabilisers, we have that $x \in G_f$, and so $f = \overline e$.
    As $p$ is assumed to be a reduced path, this is a contradiction.
    Therefore, we conclude that no non-trivial element $g\in G$ fixes a reduced two-edge path $p$ of $T$.
    
    Let $v$ be a vertex of $T$ and $e, f$ be edges with $\iota(e)=v=\iota(f)$.
    By the above paragraph, $G_e\cap G_{f}$ is trivial.
    Therefore, for all $v\in V(T)$, there are no edges $e, f\in E(T)$ such that $\iota(e)=v=\iota(f)$ and $G_e\leqslant G_f$, and so $T$ is rigid \cite[Theorem 1]{Levitt2005Characterizing}.
\end{proof}

We are now able to prove our main uniqueness result:

\begin{proposition}
\label{prop:JSJsEverywhere}
Let $G$ be a one-ended two-generator \CSA{} group, or a one-ended two-generator one-relator group with torsion which is not isomorphic to $\langle a, b\mid [a, b]^n\rangle$ for any $n>1$.

If $G$ has an essential \ZC{}-tree $T$ which contains at least one edge, then this tree is unique and is a \ZC{}-JSJ tree of $G$.
The \ZC{}-JSJ decomposition $T/G$ consists of a single rigid vertex with a single loop edge.
\end{proposition}

\begin{proof}
Let $T$ be as in the statement.
If $G$ is a one-relator group with torsion then every \ZC{}-subgroup is infinite cyclic \cite{Karrass1971subgroups, Cebotar1971subgroups}, and so every edge group of $T$ is infinite cyclic.
Recall from Proposition \ref{prop:RGZmaxJSJ} that if $G$ is a \CSA{} group then no \ZC{}-JSJ tree of $G$ has a flexible vertex.
The same is true if $G$ is a one-relator group with torsion, as here $\out(G)$ is virtually-$\mathbb{Z}$ \cite[Theorem A]{Logan2016Outer}, but if $T$ has a flexible vertex then there exists a $\mathcal{Z}$-tree $T'$ such that the quotient graph $T'/G$ has two edges, and so $\out(G)$ contains a $\mathbb{Z}^2$ subgroup \cite[Proposition 3.1.]{levitt2005automorphisms}, and so is not virtually cyclic, a contradiction.

We first prove that $T/G$ consists of a single vertex $v$ and a single loop edge.
If $G$ is a \CSA{} group then this is known \cite[Corollary 8.6]{gardam2021algebraically}, so suppose $G$ is a two-generator one-relator groups with torsion.
Here, $\out(G)$ is virtually-$\mathbb{Z}$ \cite[Theorem A]{Logan2016Outer}, and as noted above every \ZC{}-splitting is a $\mathbb{Z}$-splitting.
If $T/G$ contains more than one edge then $\out(G)$ is not virtually-$\mathbb{Z}$ \cite[Proposition 3.1.]{levitt2005automorphisms}, a contradiction.
Hence, $T/G$ contains at most one edge.
If it contained more than one vertex then $G$ would decompose as $H\ast_cK$ with $C\cong\mathbb{Z}$ and $H, K\not\cong\mathbb{Z}$, but this is not possible, as explained in the third paragraph of the proof of Lemma 5.1 of \cite{Logan2016Outer}.
Hence, $T/G$ consists of a single rigid vertex with a single loop edge.

Recall from the preamble that $T$ is reduced.
As $T$ is a \ZC{} tree, edge stabilisers are malnormal, and so by Lemma \ref{lem:fixingPaths} $T$ is rigid.
As $T$ was an arbitrary essential \ZC{} tree, we further have that $T$ is the unique essential tree in its deformation space.

We now prove that the tree $T$ is a \ZC{}-JSJ tree.
Let $T_{\mathcal{Z}_{\max}}$ be the canonical \ZC{}-JSJ tree guaranteed to exist either by Proposition \ref{prop:RGZmaxJSJ}, if $G$ is \CSA{}, or by existing theory otherwise \cite[Proposition 9.30]{Guirardel2017JSJ}.
As every vertex of $T$ is rigid, every vertex stabiliser $G_v$, $v\in V(T)$, is universally \ZC{}-elliptic, and hence elliptic in $T_{\mathcal{Z}_{\max}}$.
Therefore, $T$ dominates $T_{\mathcal{Z}_{\max}}$.
As $T_{\mathcal{Z}_{\max}}$ dominates every universally \ZC{}-elliptic tree, it follows that $T$ also dominates every universally \ZC{}-elliptic tree.
To see that $T$ is a universally \ZC{}-elliptic tree, consider again the tree $T_{\mathcal{Z}_{\max}}$.
As this is a \ZC{}-tree, if $e$ is an edge incident to an elementary vertex $v$ of $T_{\mathcal{Z}_{\max}}$ then $G_e=G_v$.
Therefore, we can slide an orbit of elementary vertices of $T_{\mathcal{Z}_{\max}}$ into an orbit of adjacent vertices (so on the level of $T_{\mathcal{Z}_{\max}}/G$, slide a single elementary vertex into an adjacent vertex), and the resulting tree is still a \ZC{}-JSJ tree of $G$, and has the same non-elementary vertex stabilisers as $T_{\mathcal{Z}_{\max}}$.
As there are finitely many orbits of elementary vertices of $T_{\mathcal{Z}_{\max}}$, this process terminates at an essential \ZC{} tree $S$ (essential as it has no elementary vertices).
As $T_{\mathcal{Z}_{\max}}$ is universally \ZC{}-elliptic and as $S$ has vertex stabilisers precisely the non-elementary vertex stabilisers of $T_{\mathcal{Z}_{\max}}$, we have that $S$ is also universally \ZC{}-elliptic.
There are then two options: either $S/G$ consists of a single rigid vertex, or $S/G$ contains at least one edge.
If $S/G$ consists of a single rigid vertex then the whole group $G$ is elliptic in every \ZC{}-splitting of $G$, and therefore $G$ has no \ZC{}-splittings, contradicting the existence of the \ZC{}-tree $T$ which has an edge.
Therefore, $S$ contains an edge, so is as in the statement, and so by the above $S/G$ consists of a single rigid vertex and a single loop edge.
As $S$ is universally-\ZC{} elliptic, it is elliptic with respect to $T$, and so there exists a tree $\hat{S}$ (called a ``standard refinement'' of $S$ \cite[Definition 2.5]{Guirardel2017JSJ}) such that every edge stabilizer of $T$ contains an edge stabilizer of $\hat{S}$ \cite[Proposition 2.2.(iii)]{Guirardel2017JSJ}, and moreover this tree is itself a \ZC{}-tree \cite[Discussion following Definition 9.25]{Guirardel2017JSJ}.
These facts combine to give us that every edge stabilizer of $T$ \emph{is} an edge stabilizer of $\hat{S}$.
Now, the tree $S$ can be obtained from $\hat{S}$ by collapsing certain edges to points \cite[Proposition 2.2.1]{Guirardel2017JSJ}. This means that the vertex groups of $S$ have graph-of-groups decompositions coming from their actions on $\hat{S}$, but as the only vertex in $S/G$ is rigid, these decompositions must be trivial. Since every vertex group of $S$ is contained in some vertex group of $T$, we see that the edges collapsed in this procedure have necessarily the same stabilisers as at least one of their endpoints, and so the collapsing does not change the set of edge stabilisers. It follows that edge stabilisers of $\hat{S}$ correspond precisely to edge stabilisers of $S$.
Therefore, every edge stabilizer of $T$ is an edge stabilizer of $S$.
As $S$ is universally \ZC{}-elliptic, $T$ is universally \ZC{}-elliptic.
Hence, $T$ is a \ZC{}-JSJ tree, as required

The result now follows, as \ZC{}-JSJ trees always lie in the same deformation space (as they dominate one another).
\end{proof}

Fuchsian groups are treated differently in the theory of JSJ decompositions. For example, Bowditch excludes these groups from his main theorem \cite[Theorem 0.1]{bowditch1998cut}. These groups play a minor role in this article (the only Fuchsian groups we consider are $\langle a, b\mid [a, b]^n\rangle$ for $n>1$), and it is easiest if, similarly to Bowditch, we simply define the \ZC{}-JSJ decomposition of such groups to be trivial.

\begin{convention}
\label{conv:JSJ}
Let $G$ be a one-ended two-generator \CSA{} group, or a one-ended two-generator one-relator group with torsion.
If $G$ is isomorphic to $\langle a, b\mid [a, b]^n\rangle$ for $n>1$, then ``the \ZC{}-JSJ decomposition of $G$'' is the graph of groups consisting a single vertex and no edges.
Otherwise, ``the \ZC{}-JSJ decomposition of $G$'' is a \ZC{}-JSJ decomposition with no elementary vertices; by Proposition \ref{prop:JSJsEverywhere} this is unique (up to conjugation).

Such a group $G$ has \emph{trivial \ZC{}-JSJ decomposition} if its \ZC{}-JSJ decomposition (as above) consists of a single vertex and no edges.
\end{convention}

\p{Spotting {\boldmath\ZC{}}-JSJ decompositions}
Our next result is only for \CSA{} groups; the analogous result for one-relator groups with torsion has a stronger statement bespoke for one-relator groups, so we postpone it until the next section.
These results give conditions which are equivalent to one of our groups having non-trivial \ZC{}-JSJ decomposition, and these conditions work well with algorithms. In particular, algorithms can easily detect Conditions (\ref{KWquote:2}) and (\ref{KWquote:3}) of the following theorem, and hence can detect a non-trivial \ZC{}-JSJ decomposition, while crucially Condition (\ref{KWquote:4}) says that Condition (\ref{KWquote:3}) actually finds, rather than just detects, the \ZC{}-JSJ decomposition of $G$.

\begin{theorem}
\label{thm:KWquote}
Let $G$ be a one-ended two-generator \CSA{} group with $\mathbb{Z}^2$ abelianisation.
The following are equivalent.
\begin{enumerate}
\item\label{KWquote:1}
$G$ has non-trivial \ZC{}-JSJ decomposition.
\item\label{KWquote:2}
$G$ has an essential $\mathbb{Z}$-splitting.
\item\label{KWquote:3}
$G$ admits a presentation $\langle a, b\mid \mathbf{t}(a, b^{-1}ab)\rangle$ for some subset $\mathbf{t}(a, y)\subset F(a, y)$.
\item\label{KWquote:4}
$G$ has \ZC{}-JSJ decomposition with a single rigid vertex and a single loop edge, corresponding to an HNN-extension with stable letter $b$:
\[
\langle a, y, b\mid \mathbf{t}(a, y), y=b^{-1}ab\rangle
\]
for some subset $\mathbf{t}(a, y)\subset F(a, y)$.
\end{enumerate}
Moreover, the presentations in (\ref{KWquote:3}) and (\ref{KWquote:4}) are related in the obvious way, so in each presentation the letters $a,b$ represent the same group elements of $G$ and the subsets $\mathbf{t}(a, y)$ are the same subsets of $F(a, y)$.
\end{theorem}

\begin{proof}
We first prove the chain of implications (\ref{KWquote:4}) $\Rightarrow$ (\ref{KWquote:2}) $\Rightarrow$ (\ref{KWquote:1}) $\Rightarrow$ (\ref{KWquote:4}).

\medskip
\noindent(\ref{KWquote:4}) $\Rightarrow$ (\ref{KWquote:2})
The splitting given by (\ref{KWquote:4}) is a $\mathbb{Z}$-splitting, and is essential as the base group $H=\langle a, y\rangle$ of the HNN-extension is not virtually-cyclic as it is one-ended \cite[Proposition 8.3]{gardam2021algebraically}.

\medskip
\noindent(\ref{KWquote:2}) $\Rightarrow$ (\ref{KWquote:1})
If $G$ has an essential $\mathbb{Z}$-splitting then $G$ is not universally \ZC{}-elliptic, i.e.\ the \ZC{}-JSJ decomposition of $G$ does not consist of a single rigid vertex with no edges.
By Proposition \ref{prop:RGZmaxJSJ}, the \ZC{}-JSJ decomposition of $G$ does not consist of a single flexible vertex with no edges.
Thus, the \ZC{}-JSJ decomposition of $G$ has an edge, as required.

\medskip
\noindent(\ref{KWquote:1}) $\Rightarrow$ (\ref{KWquote:4})
By Proposition \ref{prop:JSJsEverywhere}, the \ZC{}-JSJ decomposition of $G$ consists of a single rigid vertex $v$ and a single loop edge $e$, as claimed. This corresponds to an HNN-extension $\langle H, t\mid t^{-1}p^nt=q^m\rangle$ where $\langle p\rangle$ and $\langle q\rangle$ are maximal cyclic in $G$. As this is a \ZC{}-JSJ decomposition we have $|m|=|n|=1$, so without loss of generality we may assume $m=n=1$.
Now, there exists some $h\in H$ such that $G=\langle th, p\rangle$ and $H=\langle p, h^{-1}qh\rangle$ \cite[Proposition 8.3]{gardam2021algebraically}, and so by taking $a=p$, $b=th$ and $y=h^{-1}qh$ we obtain the following presentation for $G$, as an HNN-extension with stable letter $b$:
\[
\langle a, y, b\mid \mathbf{t}(a, y), y=b^{-1}ab\rangle
\]
This still corresponds to the \ZC{}-JSJ decomposition of $G$ as the Bass--Serre tree is the same for the two HNN-extensions, and this tree is precisely the \ZC{}-JSJ tree.
Hence, (\ref{KWquote:4}) holds.

\medskip
Finally we prove (\ref{KWquote:4}) $\Leftrightarrow$ (\ref{KWquote:3}).
\medskip

\noindent(\ref{KWquote:3}) $\Rightarrow$ (\ref{KWquote:4})
First, apply the Tietze transformation corresponding to adding the new generator $y$ as $b^{-1}ab$ to get from the presentation in (\ref{KWquote:3}) to the presentation in (\ref{KWquote:4}):
\[
\langle a, b\mid \mathbf{t}(a, b^{-1}ab)\rangle
\cong\langle a, y, b\mid \mathbf{t}(a, y), y=b^{-1}ab\rangle.
\]
As $G$ has $\mathbb{Z}^2$ abelianisation, both the subgroups $\langle a\rangle$ and $\langle y\rangle$ are maximal infinite cyclic subgroups of $G$.
Therefore, the resulting presentation describes an HNN-extension, with base group $H=\langle a, y\mid \mathbf{t}(a, y)\rangle$ and stable letter $b$, as claimed.
Now, this HNN-extension corresponds to some essential $\mathbb{Z}$-splitting, because $H$ is one-ended as $G$ is one-ended \cite[Proposition 8.3]{gardam2021algebraically}, and so by Proposition \ref{prop:JSJsEverywhere} the vertex is rigid and this splitting is in fact the \ZC{}-JSJ decomposition of $G$.
Hence, (\ref{KWquote:4}) holds.

\medskip
\noindent(\ref{KWquote:4}) $\Rightarrow$ (\ref{KWquote:3})
This is obvious, by reversing the Tietze transformation from the above equivalence.
\end{proof}


\section{One-relator groups with torsion}
\label{sec:withTorsion}
Recall that a {one relator group with torsion} is a group given by a presentation $\langle \mathbf{x}\mid S^n\rangle$ where $n>1$; such groups are always hyperbolic, and so when they are one-ended we can discuss their \ZC{}-JSJ decompositions.

In this section we prove \cref{thm:flexibilityWithTorsion}; we also prove a result about the automorphic orbits of elements in the subgroup $\langle a, b^{-1}ab\rangle$ which is used in the proofs of both Theorems \ref{thm:flexibilityHyperbolic} and \ref{thm:flexibilityWithTorsion}.

\ZC{}-JSJ decompositions are only defined for one-ended groups and Fuchsian groups must be treated with care in JSJ theory (see Convention~\ref{conv:JSJ} and the discussion preceding it). The following classification of when two-generator one-relator groups are one-ended and are Fuchsian is therefore useful. Recall that a {primitive element} of $F(a, b)$ is an element which is part of a basis for $F(a, b)$. We say a free product $A\ast B$ is \emph{trivial} if either $A$ or $B$ is trivial, and is \emph{non-trivial} otherwise.

\begin{proposition}
\label{prop:Fuchsian}
The group defined by $\langle a, b\mid S^n\rangle$, $n\geq1$ maximal, is:
\begin{enumerate}
\item\label{Fuchsian:1} one-ended if and only if $S$ is non-primitive and non-trivial.
\item\label{Fuchsian:2} one-ended and Fuchsian if and only if $n>1$ and $S$ is conjugate to $[a, b]^{\pm 1}$.
\end{enumerate}
\end{proposition}

\begin{proof}
Let $G = \langle a, b\mid S^n\rangle$, $n\geq1$ maximal.

For (\ref{Fuchsian:1}), if $S$ is trivial then $G$ is free and so not one-ended, while if $S$ is primitive then $G\cong\langle a, b\mid a^n\rangle$, so $G$ is either two-ended (if $n=1$) or infinitely ended (if $n>1$).
Hence, if $G$ is one-ended then $S$ is non-primitive and non-trivial. For the other direction, if $S$ is non-primitive and non-trivial then the group $G_1$ defined by $\langle a, b\mid S\rangle$ does not split non-trivially as a free product \cite[Proposition II.5.13]{L-S} and is not cyclic \cite[Proposition II.5.11]{L-S}, but is torsion-free \cite[Proposition II.5.18]{L-S}. Hence, $G_1$ is one-ended, and so also $G$ is one-ended \cite[Lemma 3.2]{logan2016JSJ}.

For (\ref{Fuchsian:2}), the group $G$ is one-ended and Fuchsian if and only if $n>1$ and $G\cong \langle a, b\mid [a, b]^n\rangle$ \cite[Lemma 1]{Rosenberger1986generating}, if and only if there exists an automorphism $\phi\in\aut(F(a, b))$ such that $\phi(S^n)=[a, b]^{\pm n}$ \cite{Pride1977}, if and only if $S$ is conjugate to $[a, b]^{\pm 1}$ \cite[Theorem 3.9]{mks}, as required.
\end{proof}

Next we prove the analogue of Theorem \ref{thm:KWquote} for two-generator one-relator groups with torsion. The main differences is that the set $\mathbf{t}(a, y)$ is replaced with a single word $T(a, y)^n$, and that Condition (\ref{JSJclassificationONEREL:3}) gives information about this word.
The restrictions on the word $S$ in this theorem correspond precisely to $G$ being one-ended and non-Fuchsian, by Proposition \ref{prop:Fuchsian}.
In the theorem, the pair of letters $(a, b)$ in (\ref{JSJclassificationONEREL:3.5}) and (\ref{JSJclassificationONEREL:4}) represent the same pair of group elements of $G$. However, the pair $(a,b)$ from $\mathcal{P}$ are different, and are related to those in (\ref{JSJclassificationONEREL:3.5}) and (\ref{JSJclassificationONEREL:4}) via the map $\psi$.

\begin{theorem}
\label{thm:JSJclassificationONEREL}
Let $G$ be a group admitting a two-generator one-relator presentation $\mathcal{P}=\langle a, b\mid R\rangle$ where $R= S^n$ with $n>1$ maximal and where $S\in F(a, b)$ is non-trivial, non-primitive, and not conjugate to $[a, b]^{\pm1}$.
The following are equivalent.
\begin{enumerate}
\item\label{JSJclassificationONEREL:1}
$G$ has non-trivial \ZC{}-JSJ decomposition.
\item\label{JSJclassificationONEREL:2}
$G$ has an essential $\mathbb{Z}$-splitting.
\item\label{JSJclassificationONEREL:3}
There exists an automorphism $\psi\in\aut(F(a, b))$ such that $\psi(S)\in\langle a, b^{-1}ab\rangle$.
\item\label{JSJclassificationONEREL:3.5}
$G$ admits a presentation $\langle a, b\mid T(a, b^{-1}ab)^n\rangle$ for some word $T\in F(a, y)$.
\item\label{JSJclassificationONEREL:4}
$G$ has \ZC{}-JSJ decomposition with a single rigid vertex and a single loop edge, corresponding to an HNN-extension with stable letter $b$:
\[
\langle a, y, b\mid T(a, y)^n, y=b^{-1}ab\rangle
\]
\end{enumerate}
Moreover, the presentations in (\ref{JSJclassificationONEREL:3.5}) and (\ref{JSJclassificationONEREL:4}) are related in the obvious way, so in each presentation the letters $a, b$ represent the same elements of $G$ and the words $T(a, y)$ are the same, and indeed may be taken to be such that $T(a, b^{-1}ab)=\psi(S)$, with $\psi$ as in (\ref{JSJclassificationONEREL:3}).
\end{theorem}

\begin{proof}



We start with
(\ref{JSJclassificationONEREL:3}) $\Leftrightarrow$ (\ref{JSJclassificationONEREL:3.5}).

\medskip
\noindent(\ref{JSJclassificationONEREL:3}) $\Rightarrow$ (\ref{JSJclassificationONEREL:3.5})
It is a standard result in the theory of group presentations that $\langle a, b\mid S^n\rangle\cong\langle a, b\mid \psi(S^n)\rangle$, so the result follows by taking $T(a, y)\in F(a, y)$ such that $T(a, b^{-1}ab)=\psi(S)$.

\medskip
\noindent(\ref{JSJclassificationONEREL:3.5}) $\Rightarrow$ (\ref{JSJclassificationONEREL:3})
This holds as the word $S':=T(a, b^{-1}ab)\in\langle a, b^{-1}ab\rangle$ is in the $\aut(F(a, b))$-orbit of $S$ \cite{Pride1977}, as required.

\medskip
Finally we prove (\ref{JSJclassificationONEREL:4}) $\Rightarrow$ (\ref{JSJclassificationONEREL:1}) $\Rightarrow$ (\ref{JSJclassificationONEREL:2}) $\Rightarrow$ (\ref{JSJclassificationONEREL:3.5}) $\Rightarrow$ (\ref{JSJclassificationONEREL:4}).

\medskip

\noindent(\ref{JSJclassificationONEREL:4}) $\Rightarrow$ (\ref{JSJclassificationONEREL:1})
This is clear as the stated \ZC{}-JSJ decomposition is non-trivial.

\noindent(\ref{JSJclassificationONEREL:1}) $\Rightarrow$ (\ref{JSJclassificationONEREL:2})
This is immediate from the definitions (see Convention \ref{conv:JSJ}).

\medskip
\noindent(\ref{JSJclassificationONEREL:2}) $\Rightarrow$ (\ref{JSJclassificationONEREL:3.5})
Suppose $G$ has an essential $\mathbb{Z}$-splitting.
Then $G$ has infinite outer automorphism group \cite[Theorem 5.1]{levitt2005automorphisms}, and so there exists $S'\in\langle a, b^{-1}ab\rangle$ such that $G\cong \langle a, b\mid S'\rangle$ \cite[Lemma 5.1]{Logan2016Outer}. Setting $T(a, y)$ to be the word such that $T(a, b^{-1}ab)=S'$ gives the result.

\medskip
\noindent(\ref{JSJclassificationONEREL:3.5}) $\Rightarrow$ (\ref{JSJclassificationONEREL:4})
The group $G$ has presentation
\[
\langle a, b\mid T(a, b^{-1}ab)^n\rangle\cong\langle a, b, y\mid T(a, y)^n, y=b^{-1}ab\rangle
\]
which has the claimed form. This is an HNN-extension with base group $H=\langle a, y\mid T(a, y)^n\rangle$ and stable letter $b$ as the subgroups $\langle a\rangle$ and $\langle y\rangle$ of $H$ are isomorphic, because $a$, and hence its conjugate $y=b^{-1}ab$, have infinite order in $G$ by the B.B. Newman Spelling Theorem \cite[Theorem IV.5.5]{L-S}.
Indeed, $G$ has presentations $\langle a, b\mid T(a, b^{-1}ab)^n\rangle$ and $\langle y, b\mid T(byb^{-1}, y)^n\rangle$, whence it follows that the subgroups $\langle a\rangle$ and $\langle y\rangle$ are maximal cyclic subgroups of $G$ \cite[Lemma 2.1]{Newman1973Soluble}.

Therefore, the HNN-extension describes a graph of groups decomposition $\bGamma$ of $G$ with a single vertex and a single loop edge, and where the edge groups are maximal cyclic in $G$. By Proposition \ref{prop:JSJsEverywhere}, the vertex is rigid and $\bGamma$ is the \ZC{}-JSJ decomposition of $G$, as required.
\end{proof}

Before proving Theorem \ref{thm:flexibilityWithTorsionBODYVERSION} (which corresponds to Theorem \ref{thm:flexibilityWithTorsion}) we need the following result about orbits of elements of $\langle a, b^{-1}ab\rangle$ under automorphisms of $F(a,b)$. We also use this result in the proof of Theorem \ref{thm:flexibilityHyperbolicBODYVERSION} (Theorems \ref{thm:flexibilityWithTorsionBODYVERSION} and \ref{thm:flexibilityHyperbolicBODYVERSION} combine to prove Theorem \ref{thm:flexibilityHyperbolic}).
For a letter $z\in\mathbf{z}$ and a word $W\in F(\mathbf{z})$, a \emph{$z$-syllable of $W$} is a maximal subword of the form $z^i$, $i\neq0$.

\begin{proposition}
\label{prop:whitehead}
Let $W\in\langle a, b^{-1}ab\rangle$.
Then there exists a cyclic shift $W'$ of $W$ such that $W'$ has shortest length in its $\aut(F(a, b))$-orbit, and $W'\in\langle a, b^{-1}ab\rangle$.
\end{proposition}

\begin{proof}
For a conjugacy class $[U]$ of an element $U\in F(a, b)$, we write $|[U]|:=\min\{|V|\mid V\in [U]\}$. So $|[U]|$ is simply the length of $U$ after cyclic reduction. Therefore, the proposition says in particular that for $W\in\langle a, b^{-1}ab\rangle$, there is no automorphism $\alpha\in\aut(F(a, b))$ such that $|[\alpha(W)]|< |[W]|$.
Clearly the result is true if $[W]$ contains some power of $a$. Therefore, assume that $[W]\cap \langle a\rangle=\emptyset$, and suppose that there exists an automorphism $\alpha\in\aut(F(a, b))$ such that $|[\alpha(W)]|< |[W]|$. We find a contradiction.

By the ``peak reduction'' lemma from Whitehead's algorithm \cite[Proposition I.4.20]{L-S}, there exists a Whitehead automorphism $\beta$ of the second kind (so a Whitehead automorphism which is not a permutation of $\{a, b\}^{\pm1}$) such that $|[\beta(W)]|< |[W]|$. Writing $\gamma_b\in\aut(F(a, b))$ for conjugation by $b$, the only four such Whitehead automorphisms which do not fix $[W]$ are $\beta_{1}\colon a\mapsto ab$, $b\mapsto b$, $\beta_{-1}\colon a\mapsto ab^{-1}$, $b\mapsto b$, $\gamma_b\beta_1$, and $\gamma_b\beta_{-1}$.

We find a contradiction for the automorphisms $\beta_1$ and $\beta_{-1}$; the cases of $\gamma_b\beta_1$ and $\gamma_b\beta_{-1}$ follow immediately as $[\gamma_b\beta_1(W)]=[\beta_1(W)]$ and $[\gamma_b\beta_{-1}(W)]=[\beta_{-1}(W)]$.
Consider a cyclically reduced conjugate $W'$ of $W$ which begins with $a$ or $a^{-1}$ and ends with $b$. Then $W'\in\langle a, b^{-1}ab\rangle$, and we write this word as a reduced word $U(a, b^{-1}ab)$. Let $X, Y$ denote the total number of $x$-terms, $y$-terms respectively in $U(x, y)$, and $\syl(X), \syl(Y)$ the total number of $x$-syllables, $y$-syllables respectively in $U(x, y)$.
Then $|W'|=X+Y+2\syl(Y)$. Now, $[\beta_1(W')]=[U(ba, ab)]$ and $[\beta_{-1}(W')]=[U(ab^{-1}, b^{-1}a)]$, and no free reduction or cyclic reduction happens when forming $U(ba, ab)$ or $U(ab^{-1}, b^{-1}a)$, and so $|[\beta_1(W')]|=2X+2Y=|[\beta_{-1}(W')]|$. Now, clearly $Y\geq\syl(Y)$, while $\syl(X)=\syl(Y)$ as $U(x, y)$ starts with an $x$-syllable and ends with a $y$-syllable, by our choice of $W'$, and so $X\geq\syl(Y)$. Therefore, $2X+2Y\geq X+Y+2\syl(Y)$. However, by assumption $2X+2Y< X+Y+2\syl(Y)$, so we have our promised contradiction.
\end{proof}

We now prove Theorem \ref{thm:flexibilityWithTorsion}.

\begin{theorem}[Theorem \ref{thm:flexibilityWithTorsion}]
\label{thm:flexibilityWithTorsionBODYVERSION}
Let $G$ be a group admitting a two-generator one-relator presentation $\mathcal{P}=\langle a, b\mid R\rangle$ where $R= S^n$ in $F(a, b)$ with $n>1$ maximal and where $S$ is non-trivial and not a primitive element of $F(a, b)$.
The following are equivalent.
\begin{enumerate}
\item\label{flexibilityHyperbolic:1} $G$ has non-trivial \ZC{}-JSJ decomposition.
\item\label{flexibilityHyperbolic:2} There exists a word $T$ of shortest length in the $\aut(F(a, b))$-orbit of $S$ such that $T\in\langle a, b^{-1}ab\rangle$ but $T$ is not conjugate to $[a, b]^{\pm 1}$.
\end{enumerate}
\end{theorem}

\begin{proof}
Suppose (\ref{flexibilityHyperbolic:1}) holds.
Then $G$ is not Fuchsian and so, by Proposition \ref{prop:Fuchsian}, $S$ is not conjugate to $[a, b]^{\pm1}$.
Then by Theorem \ref{thm:JSJclassificationONEREL} there exists an automorphism $\psi\in\aut(F(a, b))$ such that $\psi(S)\in\langle a, b^{-1}ab\rangle$. By Proposition \ref{prop:whitehead}, there exists a cyclic shift $T$ of $\psi(S)$ which has shortest length in its $\aut(F(a, b))$-orbit, and $T\in\langle a, b^{-1}ab\rangle$. Moreover, as $S$ is not conjugate to $[a, b]^{\pm1}$, neither is $T$ \cite[Theorem 3.9]{mks}. Hence, the word $T$ satisfies (\ref{flexibilityHyperbolic:2}).

If (\ref{flexibilityHyperbolic:2}) holds then (\ref{flexibilityHyperbolic:1}) follows from Theorem \ref{thm:JSJclassificationONEREL}.
\end{proof}


\section{Splittings and Friedl--Tillmann polytopes}
\label{sec:torsionFree}
Theorem \ref{thm:KWquote} gives a clear link between \ZC{}-JSJ decompositions and presentations of the form $\mathcal{P}'=\langle x, y\mid\mathbf{t}(x, y^{-1}xy)\rangle$ for subsets $\mathbf{t}(x, y^{-1}xy)$ of $\langle x, y^{-1}xy\rangle$.
In this section we suppose that this JSJ-presentation $\mathcal{P}'$ for a group $G$ is ``close'' to being one-relator (specifically, the generating pair $(x, y)$ also admits a one-relator presentation $\langle x, y\mid S\rangle$), and we prove that the Friedl--Tillmann polytope of $G$ is a straight line (Theorem \ref{thm:JSJFormTF}), and hence that any one-relator presentation for $G$ has this JSJ-form $\langle a, b\mid R_0(a, b^{-1}ab)\rangle$, up to a Nielsen transformation (Lemma \ref{lem:StraightLinePolytopes}).

In Section \ref{sec:RGTheorem} we prove that Theorem \ref{thm:JSJFormTF} applies to \CSA{} groups, and hence to torsion-free hyperbolic groups.
The proof of \cref{lem:StraightLinePolytopes} reviews the notion of a Friedl--Tillmann polytope.

\begin{theorem}
\label{thm:JSJFormTF}
Let $G$ be a torsion-free two-generator one-relator group with $\mathbb{Z}^2$ abelianisation, and let $\mathcal{P}=\langle a, b\mid R\rangle$ be any one-relator presentation of $G$.
Suppose there exists a normal subgroup $N$ of $F(x, y)$ such that:
\begin{enumerate}
\item\label{NEbasic:1} $F(x, y)/N\cong G$,
\item\label{NEbasic:2} $N$ can be normally generated by a single element $R_N\in F(x, y)$,
\item\label{NEbasic:3} There exists a subset $\mathbf{t}_N$ of $\langle x, y^{-1}xy\rangle$ such that $\mathbf{t}_N$ normally generates $N$.
\end{enumerate}
Then the Friedl--Tillmann polytope $\FT$ of $\mathcal{P}$ is a straight line.
\end{theorem}

\begin{proof}
Let $X$ denote the Cayley $2$-complex associated to the one-relator presentation $\langle x, y\mid R_N\rangle$ of $G$. Since $G$ is torsion-free, $X$ is aspherical \cite{Lyndon1950}. 
As $G$ is finitely presentable, there exists a finite subset $\mathbf{t}_N^{\prime}$ of $\mathbf{t}_N$ which also normally generates $N$.
Therefore, $G$ admits a finite presentation
\[\langle x, y\mid \mathbf{t}_N^{\prime}(x, y^{-1}xy), R_N\rangle,\]
and let $Y$ be the Cayley $2$-complex of this presentation.
Note that $X$ is a subcomplex of $Y$. Moreover, since every relator $r$ in $\mathbf{t}_N^{\prime}(x, y^{-1}xy)$ corresponds to a closed loop $\gamma_r$ in $X$, we can form a $3$-complex $Z$ from $Y$ by $G$-equivariantly gluing in $3$-cells, one $G$ orbit for each element in $\mathbf{t}_N^{\prime}(x, y^{-1}xy)$, in such a way that the boundary of the $3$-cell is the union of two $2$-discs glued along their boundary, where the first disc fills in the relator $r$ in $Y$, and the second fills in the loop $\gamma_r$ in $X$. It is clear that $Z$ is a $3$-dimensional cofinite $G$-complex which retracts onto $X$, and hence is aspherical.

We will now use $Z$, together with the fact that $G$ satisfies the Atiyah conjecture \cite{Jaikin-ZapirainLopez-Alvarez2018} and that it is $L^2$-acyclic \cite{DicksLinnell2007}, and compute the $L^2$-torsion polytope $P^{(2)}(G)$ of $G$, as defined by Friedl--L\"uck~\cite{FriedlLueck2017}. Note that, in general, $P^{(2)}$ is not necessarily a polytope, but a formal difference of two polytopes. Nevertheless, when $G$ is a one-relator group then $P^{(2)}(G)$ coincides with the Friedl--Tillmann polytope by \cite[Remark 5.5]{FriedlLueck2017}, and is a single polytope.

We start by looking at the cellular chain complex $C_\bullet$ of $Z$. For every $n$, the $n$-chains $C_n$ form a free $\mathbb Z G$-module.
We pick a natural cellular basis for $C_\bullet$, namely: the vertex $1$ forms the basis for $C_0$; the edges connecting $1$ to $y$ and $x$ form the basis of $C_1$; the discs filling-in $R_N$ in $X$ and the relators from $\mathbf{t}_N^{\prime}(x, y^{-1}xy)$ in $Y$ form the basis for $C_2$; and finally the $3$-cells attached to the basic discs filling-in relators from $\mathbf{t}_N^{\prime}(x, y^{-1}xy)$ form the basis of $C_3$. 

We may now identify $C_\bullet$ with
\[
 \mathbb Z G^\alpha \to \mathbb Z G^{\alpha+1} \to \mathbb Z G^2 \to \mathbb Z G
\]
where $\alpha = \vert \mathbf{t}_N^{\prime}(x, y^{-1}xy) \vert$.

\smallskip
Before proceeding any further, we need to deal with the special case in which the Fox derivatives $\frac{\partial r}{\partial x}$ are zero for all $r \in \mathbf{t}_N^{\prime}(x, y^{-1}xy)$. The fundamental formula of Fox calculus \cite{Fox1953} tells us that
\[
 \dfrac{\partial r}{\partial x}(1-x) + \dfrac{\partial r}{\partial y}(1-y) = 1-r = 0
\]
in $\mathbb Z G$.
Hence $\frac{\partial r}{\partial y}(1-y) = 0$, which forces $\frac{\partial r}{\partial y} = 0$, as $G$ satisfies the Atiyah conjecture and hence $\mathbb Z G$ has no non-trivial zero divisors. (A careful reader might observe that for this argument we need $y \neq 1$ in $G$, which is true since otherwise the abelianisation of $G$ would not be $\mathbb Z^2$.) In fact, we can conclude that $\mathbb Z G$ does not have non-trivial zero-divisors from an earlier work of Lewin--Lewin~\cite{LewinLewin1978}. Now, since both Fox derivatives of $r$ vanish, the $1$-cycle corresponding to $r$ is trivial. 

Since $R_N$ lies in the normal closure of $\mathbf{t}_N^{\prime}(x, y^{-1}xy)$, the $1$-cycle given by $R_N$ must also be trivial. Hence, the differential $C_2 \to C_1$ is zero. We know that $C_\bullet$ is $L^2$-acyclic. For torsion-free groups satisfying the Atiyah conjecture, like $G$, this amounts to saying that $\mathcal D(G) \otimes_{\mathbb Z G} C_\bullet$ is acyclic, where $\mathcal{D}(G)$ is the \emph{Linnell skew-field}, a skew-field which contains $\mathbb Z G$. In our case, if the differential $C_2 \to C_1$ is trivial, then tensoring $C_\bullet$ with $\mathcal D(G)$ yields the chain
\[
 \mathcal D(G)^\alpha \longrightarrow \mathcal D(G)^{\alpha+1} \overset{0}{\longrightarrow} \mathcal D(G)^2 \longrightarrow \mathcal D(G)
\]
which cannot be acyclic for dimension reasons. This is a contradiction.

Let $r \in \mathbf{t}_N^{\prime}(x, y^{-1}xy)$ be such that the Fox derivative $\frac{\partial r}{\partial x}$ is not zero.
Let $d$ denote the $2$-chain given by the disc filling-in $r$ in $Y$; note that $d$ is a basis element of $C_2$.
We now look at the commutative diagram with exact columns
\begin{equation}
\tag{$\dagger$}
\begin{aligned}
 \label{main comm diag}
 \xymatrix{
 0 \ar[d] \ar[r] & 0 \ar[d] \ar[r] & 0 \ar[d] \ar[r] & 0 \ar[d] \\
 0 \ar[d] \ar[r] & \langle d \rangle \ar[d] \ar[r] & \mathbb Z G^2 \ar[d]^{\mathrm{id}} \ar[r] & \mathbb Z G \ar[d]^{\mathrm{id}}\\
 \mathbb Z G^\alpha \ar[r]\ar[d]_{\mathrm{id}} & \mathbb Z G^{\alpha+1} \ar[r] \ar[d] & \mathbb Z G^2 \ar[r] \ar[d] & \mathbb Z G \ar[d] \\
 \mathbb Z G^\alpha\ar[r] \ar[d] & \mathbb Z G^{\alpha+1}/\langle d\rangle \ar[r] \ar[d] & 0 \ar[d] \ar[r] & 0 \ar[d]\\
 0 \ar[r] & 0 \ar[r] & 0 \ar[r] & 0 
 }
 \end{aligned}
\end{equation}
where $\langle d \rangle$ denotes the $\mathbb Z G$-span of $d$. Note that \eqref{main comm diag} is really a short exact sequence of chain complexes.

By assumption, the middle row of \eqref{main comm diag} is $L^2$-acyclic.
We claim that so is the second row, which we will denote by $B_\bullet$. To prove the claim, we need to look at the chain complex $B_\bullet$ in more detail:
\begin{equation*}
\begin{aligned}
\label{chain cplx top row}
 \xymatrix{
 \langle d \rangle \ar[rr]^{\Big( \dfrac{\partial r}{\partial y} \ \dfrac{\partial r}{\partial x} \Big)} & & \mathbb Z G^2 \ar[rr]^{\left( \begin{array}{c}
                                1-y \\ 1-x
                                \end{array} \right)
 } & & \mathbb Z G 
 }
\end{aligned}
\end{equation*}
Since $G$ is not cyclic, we have $1-y \neq 0$; we also have $\frac{\partial r}{\partial x} \neq 0$ by assumption. Now $B_\bullet$
fits into the exact sequence of chain complexes
\begin{equation}
\tag{$\ddagger$}
\begin{aligned}
\label{sec comm diag}
 \xymatrix{
 0 \ar[d] \ar[r] & 0 \ar[d] \ar[r] & 0 \ar[d] \\
0 \ar[d] \ar[r] & \mathbb Z G \ar[d] \ar[r]^{(1-y)} & \mathbb Z G \ar[d]^{\mathrm{id}} \\
 \langle d \rangle \ar[d]_{\mathrm{id}} \ar[r] & \mathbb Z G^2 \ar[d] \ar[r] & \mathbb Z G \ar[d]\\
 \langle d \rangle \ar[d] \ar[r]^{(\frac{\partial r}{\partial x})} & \mathbb Z G \ar[d] \ar[r] & 0 \ar[d]\\
 0 \ar[r] & 0 \ar[r] & 0 
 }
 \end{aligned}
\end{equation}
where the middle column represents the inclusion of the first coordinate into $\mathbb Z G^2$ and then the projection onto the second coordinate.
Since both horizontal differentials labelled in the diagram are multiplications by non-zero elements of $\mathbb Z G$, they are invertible over $\mathcal{D}(G)$, and hence the second and fourth rows of this commutative diagram are exact upon tensoring with $\mathcal{D}(G)$. By \cite[Lemma 2.9]{FriedlLueck2017}, the middle row also becomes exact upon tensoring with $\mathcal{D}(G)$. This proves the claim.

We are now back to examining \eqref{main comm diag}. We have just shown that the second row is $L^2$-acyclic, and hence \cite[Lemma 2.9]{FriedlLueck2017} tells us that so is the fourth row, and that
\[
 P^{(2)}(C_\bullet) = P^{(2)}(B_\bullet) + P^{(2)}(D_\bullet)
\]
where $D_\bullet$ denotes the fourth row of \eqref{main comm diag}.
We can split $P^{(2)}(B_\bullet)$ further using diagram \eqref{sec comm diag} and \cite[Lemma 1.9]{FriedlLueck2017}, and obtain 
\[
 P^{(2)}(B_\bullet) = P - Q
\]
where $-P$ is the $L^2$-torsion polytope of
\[
 \mathbb Z G \overset{1-y}{\longrightarrow} \mathbb Z G
\]
and $-Q$ of
\[
 \mathbb Z G \overset{\frac{\partial r}{\partial x}}{\longrightarrow} \mathbb Z G
\]
Now we need to recall how the polytope $P^{(2)}$ is actually constructed, at least in the above (simple) cases. When considering a chain complex of the form
\[
 \mathbb Z G \overset{z}{\longrightarrow} \mathbb Z G
\]
with $z \in \mathbb Z G \smallsetminus \{0\}$, the polytope is obtained by first taking the support $\mathrm{supp}(z) \subseteq G$, then taking the image of this set in $H_1(G;\mathbb R)$, and then taking the convex hull. At the end the polytope is given a sign, negative in this situation, and in general depending on whether the single non-trivial differential starts in an even or an odd dimension.

With the construction in mind, it is now immediate that, since $r$ is a word in $x$ and $y^{-1}xy$, the polytope $P$ is contained within the strip $[-1,0] \times \mathbb R$. The polytope $Q$ is obtained from $1-y$, and hence it is the segment $[0,1]\times\{0\}$.

Since $D_\bullet$ is concentrated around a single differential from odd to even degree chains, $P^{(2)}(D_\bullet) = -R$ where $R$ is some polytope. Hence
\[
 P^{(2)}(G) =P^{(2)}(C_\bullet) = P - Q - R
\]
Since $G$ is a one-relator group, $P^{(2)}(G)$ is actually a single polytope. Hence, in particular, $P-Q$ must be a single polytope, and therefore it must lie on the line $\{-1\} \times \mathbb R$. Subtracting a further polytope from such a polytope does not alter this fact, and hence we have finished the proof.
\end{proof}

Next, we prove the converse to Theorem \ref{thm:JSJFormTF}.
This is applied to prove that (\ref{flexibilityHyperbolic:3}) implies (\ref{flexibilityHyperbolic:2}) in the proof of Theorem \ref{thm:flexibilityHyperbolicBODYVERSION}.

\begin{lemma}
\label{lem:StraightLinePolytopes}
Let $G$ be a group admitting a two-generator one-relator presentation $\mathcal{P}=\langle a, b\mid R\rangle$ with $R\in F(a, b)'\setminus\{1\}$.
The Friedl--Tillmann polytope of $\mathcal{P}$ is a straight line if and only if there exists a word $T$ of shortest length in the $\aut(F(a, b))$-orbit of $R$ such that $T\in\langle a, b^{-1}ab\rangle$.
\end{lemma}

\begin{proof}
Suppose that the Friedl--Tillmann polytope is a line. By applying a suitable Nielsen automorphism, we may assume that $G$ is given by the presentation
$
 \langle x,y \mid S \rangle
$
and that the polytope lies on the line $\{0\} \times \mathbb R$.

To construct the Friedl--Tillmann polytope, we first define an auxiliary polytope $P$ by tracing the closed loop the word $S$ gives in the Cayley graph of the free part of the abelianisation of $G$ taken with respect to the image of $\{x,y\}$ as a generating set, and then taking the convex hull of the image of the loop in $H_1(G;\mathbb R)$. Then we observe that $P$ has to be the Minkowski sum of another polytope, say $P'$, and the square $[-1,0] \times [-1,0]$. The polytope $P'$ is precisely the Friedl--Tillmann polytope. Note that the polytope is only well-defined up to translation, so it does not matter which square we choose.
(When $G$ is torsion-free, the polytope $P'$ coincides with the $L^2$-torsion polytope constructed above.)

Now, we know that $P'$ is a line, which forces $P$ to lie inside the strip $[-1,0] \times \mathbb R$. Hence, the loop given by $S$ must also lie in this strip. Since $S \in F(x,y)'$, it can be written as a product of right conjugates of $x$ by powers of $y$. We now see that the only conjugates of $x$ appearing can be $x$ and $y^{-1}xy$. Therefore, there exists a word $T'$ in the $\aut(F(a, b))$-orbit of $R$ such that $T'\in\langle a, b^{-1}ab\rangle$, and so by Proposition \ref{prop:whitehead} there exists a word $T$ of shortest length in the $\aut(F(a, b))$-orbit of $R$ with $T\in\langle a, b^{-1}ab\rangle$, as required.

Now suppose that there exists a word $T$ of shortest length in the $\aut(F(a, b))$-orbit of $R$ such that $T\in\langle a, b^{-1}ab\rangle$. Then we may assume that $G$ is given by a presentation $\langle a, b \mid T' \rangle$ with $T'\in\langle a, bab^{-1}\rangle$. It follows that the loop given by $T$ must lie inside the strip $[-1,0] \times \mathbb R$, and so $P$ also lies inside this strip. As $P$ is the Minkowski sum of $P'$ and the square $[-1,0] \times [-1,0]$, we see that $P'$ is a straight segment.
\end{proof}

If the polytope is a single point then the conditions on the relator are much stronger. We use this lemma in the proof of Theorem \ref{thm:flexibilityHyperbolic}.

\begin{lemma}
\label{lem:SinglePointPolytopes}
Let $G$ be a group admitting a two-generator one-relator presentation $\mathcal{P}=\langle a, b\mid R\rangle$ with $R\in F(a, b)'\setminus\{1\}$.
The Friedl--Tillmann polytope of $\mathcal{P}$ is a point if and only if $R$ is conjugate to $[a, b]^k$ for some $k\in\mathbb{Z}$.
\end{lemma}

\begin{proof}
Suppose the Friedl--Tillmann polytope $P'$ of $\mathcal{P}$ is a point. As in the proof of Lemma \ref{lem:StraightLinePolytopes}, we may assume that $G$ is given by the presentation $\langle x,y \mid S \rangle$ and that the polytope lies on the line $\{0\} \times \mathbb R$. As $P'$ is a point, the polytope $P$ is the square $[0,1] \times [0, 1]$ (up to translation). This means precisely that $S = [a, b]^k$ for some non-zero $k\in\mathbb{Z}$.

Now suppose that $R$ is conjugate to $[a, b]^k$ for some $k\in\mathbb{Z}$. Then we may assume that $G$ is given by the presentation $ \langle a, b \mid (aba^{-1}b^{-1})^k \rangle$, and this presentation is easily seen to have Friedl--Tillmann polytope a point.
\end{proof}

\section{The main theorem}
\label{sec:RGTheorem}
In this section we prove our main theorem, Theorem \ref{thm:flexibilityHyperbolic}.
We require the results proven in Section \ref{sec:background}, which were for \CSA{} groups.
These results are applicable to the \BS{}-free one-relator groups of Theorem \ref{thm:flexibilityHyperbolic} as a torsion-free one-relator group is \CSA{} if and only if it is \BS{}-free \cite[Theorem B]{gardam2021algebraically}.
We start with a lemma which allows us to apply Theorem \ref{thm:JSJFormTF}.

\begin{lemma}
\label{lem:NEbasic}
Let $G$ be a torsion-free \BS{}-free two-generator one-relator group with $\mathbb{Z}^2$ abelianisation. Suppose $G$ admits an HNN-presentation $\langle H,t\mid t^{-1}p^nt=q\rangle$ where $p, q$ are nontrivial elements of $H$ and the subgroups $\langle p\rangle$ and $\langle q\rangle$ are malnormal in $H$.
Then there exist a normal subgroup $N$ of $F(x, y)$ such that:
\begin{enumerate}
\item\label{NEbasic:1} $F(x, y)/N\cong G$,
\item\label{NEbasic:2} $N$ can be normally generated by a single element $R_N\in F(x, y)$,
\item\label{NEbasic:3} There exists a finite subset $\mathbf{t}_N$ of $\langle x, y^{-1}xy\rangle$ such that $\mathbf{t}_N$ normally generates $N$.
\end{enumerate}
\end{lemma}

\begin{proof}
Suppose that $G$ admits a one-relator presentation $\langle a, b \mid R \rangle$.

The subgroups $\langle p\rangle$ and $\langle q\rangle$ are ``conjugacy separated'' in $H$, that is, for every $h\in H$ we have $h^{-1}\langle p\rangle h\cap\langle q\rangle=1$.
To see this, suppose otherwise, so $h^{-1}p^ih=q^j$ for $|i|, |j|>0$.
It follows that $h^{-1}p^ih=t^{-1}p^{nj}t$, and as $G$ is \CSA{} we have that $th^{-1}$ and $p$ share a common power, a contradiction by Britton's Lemma.

As $\langle p\rangle$ and $\langle q\rangle$ are conjugacy separated in $H$, there exists a Nielsen transformation $\psi$ of $F(a, b)$ and some $g\in G$ such that $\psi(a)=_Gg^{-1}p^ig$ and $\psi(b)=_Gg^{-1}thg$ for some $i\in \mathbb{Z}$ and some $h\in H$ \cite[Corollary 3.1]{Kapovich1999TwoGen}. Therefore, writing $x:=p$ and $y:=th$, the pair $(x, y)$ generates $G$. But also the pair $(x^i,y)$ generates $G$, and as $G$ has abelianisation $\mathbb{Z}^2$ we have that $|i|=1$.
Note that as the generating pair $(a, b)$ of $G$ admits a one-relator presentation, and as $\psi$ is a Nielsen transformation, we have that the pair $(x, y)$ also admits a one-relator presentation $\langle x, y\mid R_N\rangle$. Hence, points (\ref{NEbasic:1}) and (\ref{NEbasic:2}) of the theorem hold for some subgroup $N\leq F(x, y)$.

We now establish (\ref{NEbasic:3}) for $N\leq F(x, y)$. Firstly, note that $H=\langle p, h^{-1}qh\rangle$ by \cite[Proposition 8.3]{gardam2021algebraically}, and so $H$ has presentation $\langle x, z\mid\mathbf{s}(x, z)\rangle$ where $x:=p$ (as above) and $z:=h^{-1}qh$.
    Therefore, $G$ has presentation $\langle x, z, t\mid \mathbf{s}(x, z), t^{-1}x^nt=w(x, z)zw(x, z)^{-1}\rangle$ where $w(x, z)$ represents the element $h\in H$. As remarked above the theorem, because $G$ is finitely presentable we may assume for our purposes with $G$ (but not $H$) that the set $\mathbf{s}(x, z)$ is finite. As $y=th$, we can apply Tietze transformations as follows, where the generators $x$ and $y$ correspond precisely to the generators $x$ and $y$ in the previous paragraph:
\begin{align*}
&\langle x, z, t\mid \mathbf{s}(x, z), t^{-1}x^nt=w(x, z)zw^{-1}(x, z)\rangle\\
&\cong\langle x, y, z, t\mid \mathbf{s}(x, z), t^{-1}x^nt=w(x, z)zw^{-1}(x, z), y=tw(x, z)\rangle\\
&\cong\langle x, y, z, t\mid \mathbf{s}(x, z), y^{-1}x^ny=z, y=tw(x, z)\rangle\\
&\cong\langle x, y, z\mid \mathbf{s}(x, z), y^{-1}x^ny=z\rangle\\
&\cong\langle x, y\mid \mathbf{s}(x, y^{-1}x^ny)\rangle\\
&\cong\langle x, y\mid \mathbf{t}_N\rangle
\end{align*}
for some finite subset $\mathbf{t}_N = \mathbf{s}(x, y^{-1}x^ny)$ of $\langle x, y^{-1}xy\rangle$.
As the symbols $x$ and $y$ in this presentation of $G$ and the one-relator presentation $\langle x, y\mid R_N\rangle$ of $G$ both represent the same elements of $G$, it follows that $N$ is the normal closure of the set $\mathbf{t}_N$ in $F(x, y)$, as required.
\end{proof}

Our next result is the torsion-free part of our main theorem. It is worth stating separately because it does not include the exceptional cases of points (\ref{flexibilityHyperbolic:2}) and (\ref{flexibilityHyperbolic:3}).

\begin{theorem}
\label{thm:flexibilityHyperbolicRGVERSION}
Let $G$ be a torsion-free \BS{}-free group admitting a two-generator one-relator presentation $\mathcal{P}=\langle a, b\mid R\rangle$ with $R\in F(a, b)'\setminus\{1\}$.
The following are equivalent.
\begin{enumerate}
\item\label{flexibilityHyperbolic:1}
$G$ has non-trivial \ZC{}-JSJ decomposition.
\item\label{flexibilityHyperbolic:2} There exists a word $T$ of shortest length in the $\aut(F(a, b))$-orbit of $R$ such that $T\in\langle a, b^{-1}ab\rangle$.
\item\label{flexibilityHyperbolic:3} The Friedl--Tillmann polytope of $\mathcal{P}$ is a straight line.
\end{enumerate}
\end{theorem}

\begin{proof}
Note that $G$ is one-ended by Proposition \ref{prop:Fuchsian}, and is a \CSA{} group. Therefore, Theorem \ref{thm:KWquote} is applicable.

Items (\ref{flexibilityHyperbolic:2}) and (\ref{flexibilityHyperbolic:3}) are equivalent by Lemma \ref{lem:StraightLinePolytopes}.

\noindent(\ref{flexibilityHyperbolic:2}) $\Rightarrow$ (\ref{flexibilityHyperbolic:1})
The word $T$ gives a presentation for $G$:
\[
\langle a, b\mid T\rangle
=\langle a, b\mid T_0(a, b^{-1}ab)\rangle.
\]
Then (\ref{flexibilityHyperbolic:1}) follows by Theorem \ref{thm:KWquote}.

\medskip
\noindent(\ref{flexibilityHyperbolic:1}) $\Rightarrow$ (\ref{flexibilityHyperbolic:3})
By Theorem \ref{thm:KWquote} the conditions of Lemma \ref{lem:NEbasic} are satisfied, so we may apply Theorem \ref{thm:JSJFormTF}, and (\ref{flexibilityHyperbolic:3}) follows.
\end{proof}

We wish to combine Theorem \ref{thm:flexibilityHyperbolicRGVERSION} and Theorem \ref{thm:flexibilityWithTorsion} to prove Theorem \ref{thm:flexibilityHyperbolic}.
However, if $R=S^n$ with $n>1$ in Theorem \ref{thm:flexibilityHyperbolic} then we cannot immediately apply Theorem \ref{thm:flexibilityWithTorsion}, as this theorem deals with the root $S$ rather than the relator $R=S^n$. The following observation is therefore useful.
\begin{lemma}
\label{lem:AutsAndPowers}
If $n>1$ and $S\in F(a, b)$, then $S$ has shortest length in its $\aut(F(a, b))$-orbit if and only if $S^n$ has shortest length in its $\aut(F(a, b))$-orbit.
\end{lemma}

\begin{proof}
Suppose that $\alpha, \beta \in \aut(F(a, b))$ are such that $\alpha(S)$ and $\beta(S^n)$ are shortest-length in their respective $\aut(F(a, b))$-orbits.
In particular, they are both cyclically reduced and since $\beta(S^n) = \beta(S)^n$ is cyclically reduced if and only if $\beta(S)$ is cyclically reduced, we see that $|\alpha(S)^n| = n |\alpha(S)|$ and $|\beta(S)^n| = n |\beta(S)|$.
Now by the shortest-length assumptions we have \[
    |\alpha(S^n)| \geq |\beta(S^n)| = n |\beta(S)| \geq n |\alpha(S)| = |\alpha(S^n)|
\] so equality holds.
The conclusion follows by taking $\alpha$ or $\beta$ to be the identity as appropriate.
\end{proof}

We now combine Theorem \ref{thm:flexibilityHyperbolicRGVERSION} and Theorem \ref{thm:flexibilityWithTorsion} to prove Theorem \ref{thm:flexibilityHyperbolic}.

\begin{theorem}
[Theorem \ref{thm:flexibilityHyperbolic}]
\label{thm:flexibilityHyperbolicBODYVERSION}
Let $G$ be a \BS{}-free group admitting a two-generator one-relator presentation $\mathcal{P}=\langle a, b\mid R\rangle$ with $R\in F(a, b)'\setminus\{1\}$.
The following are equivalent.
\begin{enumerate}
\item\label{flexibilityHyperbolic:1}
$G$ has non-trivial \ZC{}-JSJ decomposition.
\item\label{flexibilityHyperbolic:2} There exists a word $T$ of shortest length in the $\aut(F(a, b))$-orbit of $R$ such that $T\in\langle a, b^{-1}ab\rangle$ but $T$ is not conjugate to $[a, b]^{k}$ for any $k\in\mathbb{Z}$.
\item\label{flexibilityHyperbolic:3} The Friedl--Tillmann polytope of $\mathcal{P}$ is a straight line, but not a single point.
\end{enumerate}
\end{theorem}

\begin{proof}
If there exists a word $T$ in the $\aut(F(a, b))$-orbit of $R$ which is conjugate to $[a, b]^{k}$ for some $k\in\mathbb{Z}$, then $G$ is either abelian (when $|k|=1$) or contains torsion (when $k>1$). Applying this to Lemma \ref{lem:SinglePointPolytopes}, we also have that if $\mathcal{P}$ is a single point then $G$ is either abelian or contains torsion. Therefore, if $G$ is torsion-free then the result follows by Theorem \ref{thm:flexibilityHyperbolicRGVERSION}.

Suppose $G$ contains torsion, so $R=S^n$ in $F(a, b)$ for some $n>1$ maximal.

\medskip
\noindent(\ref{flexibilityHyperbolic:2}) $\Leftrightarrow$ (\ref{flexibilityHyperbolic:3})
This follows from Lemmas \ref{lem:StraightLinePolytopes} and \ref{lem:SinglePointPolytopes}.

\medskip
\noindent(\ref{flexibilityHyperbolic:1}) $\Rightarrow$ (\ref{flexibilityHyperbolic:2})
If $G$ has non-trivial \ZC{}-JSJ decomposition then, by Theorem \ref{thm:flexibilityWithTorsion}, there exists a word $T'$ of shortest length in the $\aut(F(a, b))$-orbit of $S$ such that $T'\in\langle a, b^{-1}ab\rangle$ but $T'$ is not conjugate to $[a, b]^{\pm 1}$. Then the word $T:=(T')^n$ has shortest length in its $\aut(F(a, b))$-orbit, by Lemma \ref{lem:AutsAndPowers},
and hence in the $\aut(F(a, b))$-orbit of $R$, is contained in $\langle a, b^{-1}ab\rangle$, and is not conjugate to $[a, b]^{\pm n}$ (we are using uniqueness of roots in free groups here). As $n$ is maximal, $T$ is not conjugate to any other power of $[a, b]^{\pm1}$, and so (\ref{flexibilityHyperbolic:2}) holds.

\medskip
\noindent(\ref{flexibilityHyperbolic:2}) $\Rightarrow$ (\ref{flexibilityHyperbolic:1})
For $T$ as in (\ref{flexibilityHyperbolic:2}) , write $\psi\in\aut(F(a, b))$ for the automorphism such that $\psi(R)=T$, and set $T':=\psi(S)$. Then $T'$ has shortest length in the $\aut(F(a, b))$-orbit of $S$, by Lemma \ref{lem:AutsAndPowers},
is in $\langle a, b^{-1}ab\rangle$ and is not conjugate to $[a, b]^{\pm 1}$. Therefore, by Theorem \ref{thm:flexibilityWithTorsion}, (\ref{flexibilityHyperbolic:1}) holds.
\end{proof}

\section{Consequences of the main theorem}
\label{sec:mainProofs}

We now prove Corollaries \ref{corol:JSJform}--\ref{corol:OutCommensurability} from the introduction

\p{Forms of \boldmath{\ZC{}-JSJ} decompositions}
We first prove Corollary \ref{corol:JSJform}, which describes the \ZC{}-JSJ decompositions of our groups.

\begin{corollary}[Corollary \ref{corol:JSJform}]
\label{corol:JSJformBODYVERSION}
Let the group $G$ and the presentation $\mathcal{P}$ be as in Theorem \ref{thm:flexibilityHyperbolic} or \ref{thm:flexibilityWithTorsion}, and write $R=S^n$ for $n\geq1$ maximal.
Suppose that $G$ has non-trivial \ZC{}-JSJ decomposition $\mathbf{\Gamma}$. Then the graph underlying $\mathbf{\Gamma}$ consists of a single rigid vertex and a single loop edge.
Moreover, the corresponding HNN-extension has vertex group $\langle a, y\mid T_0^n(a, y)\rangle$, stable letter $b$, and attaching map given by $y=b^{-1}ab$, where $T_0(a, b^{-1}ab)$ is the word $T$ from the respective theorem. Finally, $|T_0|<|S|$.
\end{corollary}

\begin{proof}
By (\ref{flexibilityHyperbolic:2}) of Theorem \ref{thm:flexibilityHyperbolic} or \ref{thm:flexibilityWithTorsion}, as appropriate, $G$ admits a presentation of the form $\langle a, b\mid T_0^n(a, b^{-1}ab)\rangle$, $n\geq1$, where $T_0^n(a, b^{-1}ab)$ and $T^n(a, b)$ are the same words. Hence, $G$ has a presentation 
\[\langle a, b, y\mid T_0^n(a, y), b^{-1}ab=y\rangle,\]
which is an HNN-extension with stable letter $b$, and corresponds to a graph of groups decomposition $\bGamma$ as in the statement of the corollary.
This decomposition corresponds to the \ZC{}-JSJ decomposition of $G$ and has rigid vertex by Theorem \ref{thm:KWquote}, if $G$ is as in Theorem \ref{thm:flexibilityHyperbolic}, and by Theorem \ref{thm:JSJclassificationONEREL} otherwise.

Finally, as $T$ is not a power of a primitive we have that $|T_0|<|T|$, and so $|T_0|<|T|\leq|S|$, as required.
\end{proof}

\p{Algorithmic consequences}
We now prove Corollary \ref{corol:JSJcomp}, on computing \ZC{}-JSJ decompositions. Note that it is possible to calculate a precise bound for the time complexity of this algorithm, as the constants are known for the time complexity of Whitehead's algorithm in $F(a, b)$ \cite{Khan2004whitehead} \cite[Theorem 1.1]{Cooper2015Classification}.

\begin{corollary}[Corollary \ref{corol:JSJcomp}]
\label{corol:JSJcompBODYVERSION}
There exists an algorithm with input a presentation $\mathcal{P}=\langle a, b\mid R\rangle$ of a group $G$ from Theorem \ref{thm:flexibilityHyperbolic} or \ref{thm:flexibilityWithTorsion}, and with output the \ZC{}-JSJ decomposition for $G$.

This algorithm terminates in $O(|R|^2)$-steps.
\end{corollary}

\begin{proof}
Write $R$ as $S^n$ with $n\geq1$ maximal.
Write $\mathcal{O}_S$ for the set of shortest elements in the $\aut(F(a, b))$-orbit of $S$.

Firstly, use Whitehead's algorithm to compute the set $\mathcal{O}_S$; this takes $O(|S|^2)\leq O(|R|^2)$ steps \cite{Khan2004whitehead}.
Then compute $\mathcal{T}:=\mathcal{O}_S\cap\langle a, b^{-1}ab\rangle$; Khan gave a linear bound on the cardinality of $\mathcal{O}_S$ \cite{Khan2004whitehead} and we can (constructively) test membership of any fixed finitely generated subgroup of a free group in time linear in the length of the input word, so this takes $O(|S|^2)\leq O(|R|^2)$ steps.
If the intersection $\mathcal{T}$ is empty then, by Theorem \ref{thm:flexibilityHyperbolic} or \ref{thm:flexibilityWithTorsion} as appropriate, the \ZC{}-JSJ decomposition of $G$ is trivial. Therefore, output as the \ZC{}-JSJ decomposition of $G$ the graph of groups consisting of a single vertex, with vertex group $G$, and no edges.

If $\mathcal{T}$ is non-empty then take some $T\in\mathcal{T}$.
Rewrite this word $T$ as a word $T_0(a, b^{-1}ab)$; this takes $O(|T|)\leq O(|R|)$ steps.
By Corollary \ref{corol:JSJform}, the \ZC{}-JSJ decomposition of $G$ is the graph of groups with a single vertex, a single loop edge, vertex group $\langle a, y\mid T_0^n(a, y)\rangle$ and the attaching map $a=byb^{-1}$, and so output this as the \ZC{}-JSJ decomposition of $G$.
\end{proof}

We next prove Corollary \ref{corol:JSJdetect}, on detecting non-trivial \ZC{}-JSJ decompositions.

\begin{corollary}[Corollary \ref{corol:JSJdetect}]
\label{corol:JSJdetectBODYVERSION}
There exists an algorithm with input a presentation $\mathcal{P}=\langle a, b\mid R\rangle$ of a group $G$ from Theorem \ref{thm:flexibilityHyperbolic}, and with output {\ttfamily\upshape yes} if the group $G$ has non-trivial \ZC{}-JSJ decomposition and {\ttfamily\upshape no} otherwise.

This algorithm terminates in $O(|R|)$-steps.
\end{corollary}

\begin{proof}
We describe an algorithm to decide in $O(|R|)$ steps whether or not the Friedl--Tillmann polytope is a line; the result then follows from Theorem \ref{thm:flexibilityHyperbolic}.
We need to avoid computing the entire convex hull, because that cannot in general be done in linear time.
Suppose that the Friedl--Tillmann polytope $P$ is a line, namely from $(x_0, y_0)$ to $(x_1, y_1)$ with $x_0 \leq x_1$.
Then the Minkowski sum $P'$ of $P$ and $[0, 1] \times [0, 1]$ will be the convex hull of $6$ points (with only $4$ needed in the vertical or horizontal case).
If $y_0 \leq y_1$, the points are $(x_0+1, y_0), (x_0, y_0), (x_0, y_0+1)$ in the bottom left corner and $(x_1, y_1+1), (x_1+1, y_1+1), (x_1+1, y_1)$ in the top right corner, and the minimal bounding box (axis-parallel rectangle) has opposite corners $(x_0, y_0)$ and $(x_1+1, y_1+1)$.
If $y_0 > y_1$, the minimal bounding box instead has opposite corners $(x_0, y_0+1)$ and $(x_1+1, y_1)$

The algorithm proceeds as follows.
If necessary, we put the word into reduced form in linear time using a stack (for each letter, if it is inverse to the top of the stack then discard it and pop from the stack, else push it onto the stack).
We then determine two candidate polytopes $P_0'$ and $P_1'$, each on $6$ points, by computing the minimal bounding box of (the convex hull of) the path $\gamma$ by walking the path and keeping track of the minimal and maximal $x$ and $y$ coordinates visited, then applying the two possibilities for $P'$ in terms of the bounding box discussed in the previous paragraph.
It remains to check if the convex hull of $\gamma$ is contained in and contains some candidate.
For each candidate, we now walk the path again and check that we stay inside the polytope (a constant time check, as it is on $6$ points) and visit the $6$ defining points.
If either candidate matches, we have determined that the Friedl--Tillman polytope is a line, otherwise it is not.
\end{proof}

Finally, we prove Corollary \ref{corol:Outdetect}, on determining the commensurability classes of outer automorphism groups.
This result is only applicable to hyperbolic groups, as it relies on a connection between \ZC{}-JSJ decompositions and outer automorphism groups which is only known for hyperbolic groups.

\begin{corollary}[Corollary \ref{corol:Outdetect}]
\label{corol:OutdetectBODYVERSION}
There exists an algorithm with input a presentation $\mathcal{P}=\langle a, b\mid R\rangle$ of a hyperbolic group $G$ from Theorem \ref{thm:flexibilityHyperbolic} or \ref{thm:flexibilityWithTorsion} that determines which one of the following three possibilities holds: the outer automorphism group of $G$ is finite, is virtually $\mathbb{Z}$, or is isomorphic to $\operatorname{GL}_2(\mathbb{Z})$.

If $G$ is as in Theorem \ref{thm:flexibilityHyperbolic}, this algorithm terminates in $O(|R|)$-steps. Else, it terminates in $O(|R|^2)$-steps.
\end{corollary}

\begin{proof}
If $R$ is conjugate to $[a, b]^{n}$ for some $n\in\mathbb{Z}$ then $\out(G)\cong\operatorname{GL}_2(\mathbb{Z})$ \cite[Theorem A]{Logan2016Outer}. Else if $G$ has non-trivial \ZC{}-JSJ decomposition then, by Theorems \ref{thm:KWquote} and \ref{thm:JSJclassificationONEREL}, $G$ has an essential $\mathbb{Z}$-splitting, and so $\out(G)$ is virtually $\mathbb{Z}$ \cite[Theorem 5.1]{levitt2005automorphisms}. Else, by Theorems \ref{thm:KWquote} and \ref{thm:JSJclassificationONEREL}, $G$ admits no essential $\mathbb{Z}$-splittings, and so $\out(G)$ is finite \cite[Theorem 5.1]{levitt2005automorphisms}.

Our algorithm is therefore as follows:
Firstly, determine whether $R$ is conjugate to $[a, b]^n$ for some $n\in\mathbb{Z}$; this takes $O(|R|)$ steps. If it is, then $\out(G)\cong\operatorname{GL}_2(\mathbb{Z})$. Otherwise determine whether $G$ has trivial \ZC{} JSJ decomposition; this takes $O(|R|)$ steps if $R\in F(a, b)'$, by Corollary \ref{corol:JSJdetect}, and otherwise takes $O(|R|^2)$ steps, by Corollary \ref{corol:JSJcomp}. If $G$ has non-trivial \ZC{}-JSJ decomposition $\out(G)$ is virtually $\mathbb{Z}$, else $\out(G)$ is finite.
\end{proof}

\p{Relationships between outer automorphism groups}
We now prove Corollary \ref{corol:OutCommensurability}, on embeddings of outer automorphism groups.
As with the above, this requires the groups involved to be hyperbolic.

\begin{corollary}[Corollary \ref{corol:OutCommensurability}]
\label{corol:OutCommensurabilityBODYVERSION}
Write $G_k$ for the group defined by $\langle a, b\mid S^k\rangle$, where $k\geqslant1$ is maximal. If $S\in F(a, b)' \setminus \{1\}$ and $G_1$ is hyperbolic then:
\begin{enumerate}
\item\label{OutCommensurability:1} $\out(G_m)\cong\out(G_n)$ for all $m, n>1$.
\item\label{OutCommensurability:2} $\out(G_n)$ embeds with finite index in $\out(G_1)$.
\end{enumerate}
\end{corollary}

\begin{proof}
We first prove (\ref{OutCommensurability:1}). For $k>1$, every automorphism $\phi$ of $G_k$ is induced by a Nielsen transformation $\phi_t$ of $F(a, b)$ \cite[Principal Lemma]{Pride1977}. This gives a homomorphism $\theta_k:H_k\twoheadrightarrow \out(G_k)$ for some subgroup $H_k$ of $\out(F(a, b))$. Now, such a map $\phi_t$ sends $S$ to a conjugate of $S$ or of $S^{-1}$ in $F(a,b)$ \cite[Theorem N5]{mks}, so the maps are dependent only on the word $S$, and independent of the exponent $k$ in $S^k$, and therefore $H_m=H_n$ for all $m, n>1$. Moreover, the maps $\theta_k$ are all isomorphisms \cite[Theorem 4.2]{Logan2016Outer}, and so we have $\out(G_m)\cong H_m=H_n\cong\out(G_n)$ as required.

For (\ref{OutCommensurability:2}), as we observed in the previous paragraph, the automorphisms of $G_n$ depend solely on $S$ and so we have a sequence of homomorphisms $H_n\xrightarrow{\theta_n}\out(G_n)\xrightarrow{\xi_n} \out(G_1)$ where $\theta_n$ is injective (it is an isomorphism in fact). Then $\xi_n\theta_n$ is injective, by an identical proof to the proof that $\theta_n$ is injective \cite[Lemma 4.4]{Logan2016Outer}, and so $\xi_n$ is injective. To see that $\im(\xi_n)$ has finite index in $\out(G_1)$ note that because $G_1$ is hyperbolic, $S$ is not conjugate to $[a, b]^{\pm1}$, and so $G_1$ has non-trivial \ZC{}-JSJ decomposition if and only if $G_n$ has non-trivial \ZC{}-JSJ decomposition by Theorems \ref{thm:flexibilityHyperbolic} and \ref{thm:flexibilityWithTorsion}.
If both \ZC{}-JSJ decompositions are non-trivial then, by Corollary \ref{corol:JSJform}, both \ZC{}-JSJ decomposition are essential $\mathbb{Z}$-splittings consisting of a single rigid vertex with a single loop edge, and so both $\out(G_1)$ and $\out(G_n)$ are virtually $\mathbb{Z}$ \cite[Theorem 5.1]{levitt2005automorphisms}, and the result follows.
If both \ZC{}-JSJ decompositions are trivial then both $\out(G_1)$ and $\out(G_n)$ are finite \cite[Theorem 5.1]{levitt2005automorphisms}, and the result follows.
\end{proof}

The above proof uses the fact that $S\in F(a, b)'$ in two places; firstly to prove that $\out(G_n)$ embeds in $\out(G_1)$ (which is hidden in the citation), and secondly to obtain finite index (by using Theorem \ref{thm:flexibilityHyperbolic}).
If $S\not\in F(a, b)'$ then we can still ask if $\out(G_n)$ embeds in $\out(G_1)$. Indeed, we still obtain a sequence $H_n\xrightarrow{\theta_n}\out(G_n)\xrightarrow{\xi_n} \out(G_1)$, but here the proof that $\theta_n$ is injective does not extend to the map $\xi_n\theta_n$ \cite[Lemma 4.6]{Logan2016Outer} (as the proof uses the theory of one-relator groups with torsion, in the form of a strengthened version of the B.B. Newman Spelling Theorem).
Therefore, we are currently unable to prove that $\out(G_n)$ embeds in $\out(G_1)$ in general.

\bibliographystyle{amsalpha}
\bibliography{BibTexBibliography}
\end{document}